\documentclass[a4paper,11pt]{amsart}

\usepackage{amscd,amssymb}
\usepackage{mathrsfs}
\usepackage{color}
\usepackage{multirow,bigdelim}
\usepackage{bm}
\usepackage[all]{xy}
\usepackage[driverfallback=dvipdfm]{hyperref}
\usepackage{geometry}
\newtheorem{thm}{Theorem}[section]
\newtheorem{lem}[thm]{Lemma}
\newtheorem{cor}[thm]{Corollary}
\newtheorem{prop}[thm]{Proposition}
\newtheorem{prob}[thm]{Problem}
\newtheorem{claim}[thm]{Claim}

\theoremstyle{definition}
\newtheorem{defin}[thm]{Definition}
\newtheorem{eg}[thm]{Example}
\newtheorem*{conventions}{Conventions}
\newtheorem*{notation}{Notation}
\newtheorem*{acknowledgment}{Acknowledgment}

\theoremstyle{remark}
\newtheorem{rem}[thm]{Remark}

\numberwithin{equation}{section}

\newcommand{\bA}{\mathbb{A}}
\newcommand{\bC}{\mathbb{C}}

\newcommand{\bF}{\mathbb{F}}
\newcommand{\bk}{\Bbbk}
\newcommand{\kc}{\overline{\Bbbk}}
\newcommand{\sO}{\mathscr{O}}
\newcommand{\bP}{\mathbb{P}}
\newcommand{\bQ}{\mathbb{Q}}
\newcommand{\bR}{\mathbb{R}}
\newcommand{\bZ}{\mathbb{Z}}
\newcommand{\cont}{\mathrm{cont}}
\newcommand{\Cl}{\mathrm{Cl}}
\newcommand{\Gal}{\mathrm{Gal}(\overline{\Bbbk}/{\Bbbk})}
\newcommand{\Pic}{\mathrm{Pic}}
\newcommand{\Proj}{\mathrm{Proj}}
\newcommand{\red}{\mathrm{red}}
\newcommand{\Sing}{\mathrm{Sing}}
\newcommand{\Spec}{\mathrm{Spec}}
\newcommand{\Supp}{\mathrm{Supp}}

\newcommand{\wV}{\widetilde{V}}
\newcommand{\cV}{\check{V}}
\newcommand{\wD}{\widetilde{D}}
\newcommand{\cD}{\check{D}}
\newcommand{\wS}{\widetilde{S}}
\newcommand{\hS}{\hat{S}}
\newcommand{\cS}{\check{S}}
\newcommand{\wDelta}{\widetilde{\Delta}}
\newcommand{\hDelta}{\hat{\Delta}}
\newcommand{\cDelta}{\check{\Delta}}
\newcommand{\wC}{\widetilde{C}}
\newcommand{\hC}{\hat{C}}
\newcommand{\cC}{\check{C}}
\newcommand{\wE}{\widetilde{E}}
\newcommand{\hE}{\hat{E}}
\newcommand{\cE}{\check{E}}

\newcommand{\wGamma}{\widetilde{\Gamma}}

\begin{document}

\title[Compactifications of $\mathbb{A}^2$ over non-closed fields]{Compactifications of the affine plane over non-closed fields}

%    Remove any unused author tags.

%    author one information
\author{}
\address{}
\curraddr{}
\email{}
\thanks{}

%    author two information
\author{Masatomo Sawahara}
\address{Graduate School of Science and Engineering, Saitama University, Shimo-Okubo 255, Sakura-ku Saitama-shi,  Saitama 338-8570, JAPAN}
\curraddr{}
\email{sawahara.masatomo@gmail.com}
\thanks{}



\date{}

\dedicatory{}

\begin{abstract}
In this article, we give the classification of normal del Pezzo surfaces of rank one with at most log canonical singularities containing the affine plane defined over an algebraically non-closed field of characteristic zero. 
As an application, we have a criterion for del Pezzo fibrations with canonical singularities whose generic fibers are not smooth to contain vertical $\bA ^2$-cylinders. 
\end{abstract}

\maketitle
\setcounter{tocdepth}{1}
%%%%%%%%%%%%%%%%%%%%%%%%%%%%%%%%%%%%%%%%%%%%%%%%%%%%%%%%%%%%%%%%%%%%%%%%%%%%%%%%%%%%%%%%%%%%%%%%%%%%%%%%%%
Unless otherwise stated, let $\bk$ be a field of characteristic zero (not necessarily an algebraically closed field) and let $\kc$ be the algebraic closure of $\bk$. 
%%%%%%%%%%%%%%%%%%%%%%%%%%%%%%%%%%%%%%%%%%%%%%%%%%%%%%%%%%%%%%%%%%%%%%%%%%%%%%%%%%%%%%%%%%%%%%%%%%%%%%%%%%
\section{Introduction}\label{1}
%%%%%%%%%%%%%%%%%%%%%%%%%%%%%%%%%%%%%%%%%%%%%%%%%%%%%%%%
\subsection{Cylinders in MFS}
%%%%%%%%%%%%%%%%%
Let $X$ be an algebraic variety over $\bk$. 
An open subset $U$ of $X$ is called an {\it $\bA ^r_{\bk}$-cylinder} if $U$ is isomorphic to $\bA ^r_{\bk} \times Z$ for some variety $Z$ over $\bk$. 
When the rank $r$ of cylinder $U$ is not important, $U$ is just said to be a {\it cylinder} in what follows. 
Certainly, cylinders are geometrically simple objects, however, they receive a lot of attention recently from the viewpoint of unipotent group actions on affine cones over polarized varieties ({\cite{KPZ11}}). 
Thus, it is important to find a cylinder in projective varieties. 
According to {\cite{DK19}}, this attempt is essentially reduced to finding a cylinder contained in Mori Fiber Space (MFS, for short). 
In particular, as a special and an ideal situation, we shall consider finding a {\it vertical $\bA ^r_{\bC}$-cylinder} with respect to MFS $f:X \to Y$ defined over the complex number field $\bC$ (see {\cite{DK18}} for definition). 
Notice that the existence of a vertical $\bA ^r_{\bC}$-cylinder with respect to $f$ is equivalent to saying that the generic fiber $X_{\eta}$ of $f$, which is a Fano variety of rank one defined over the field $\bC (Y)$ of the rational functions on $Y$, contains an $\bA ^r_{\bC (Y)}$-cylinder (see {\cite[Lemma 3]{DK18}}). 
In particular, the dimension of $X_\eta$ over $\bC (Y)$ is less than that of $X$ unless $X$ is a rank one Fano variety. 
However, the base field $\bC (Y)$ is not algebraically closed unless $Y$ is a point. 
Hence, in order to find a vertical cylinder with respect to MFS, the following problem is essential to consider: 
%%%%%%%%%%%%%%%%%
\begin{prob}\label{prob}
Let $V$ be a normal Fano variety of rank one over a field $\bk$ of characteristic zero, and let $r$ be an integer with $1 \le r \le \dim V$. 
In which case does $V$ contain an $\bA ^r_{\bk}$-cylinder? 
\end{prob}
%%%%%%%%%%%%%%%%%
Let $V$ be a Fano variety of rank one over $\bk$. 
In the case of $\dim V =1$, the above problem is quite easy. 
Indeed, $V$ contains an $\bA ^1_{\bk}$-cylinder if and only if $V$ has a $\bk$-rational point. 
In case of $\dim V=2$, then $V$ is a del Pezzo surface. 
If $V$ is smooth, Dubouloz and Kishimoto showed that $V$ contains an $\bA ^1_{\bk}$-cylinder (resp. an $\bA ^2_{\bk}$-cylinder) if and only if $V$ has a $\bk$-rational point and is of degree greater than or equal to $5$ (resp. $8$) (see {\cite{DK18}}). 
Notice that any $\bA ^2_{\bk}$-cylinder in a projective surface over $\bk$ is clearly isomorphic to the affine plane $\bA ^2_{\bk}$. 
Recently, the author generalizes their work to normal del Pezzo surfaces. 
More precisely, the existing condition of $\bA ^1_{\bk}$-cylinders in a normal del Pezzo surface $V$ with at most canonical singularities is completely determined only by means of the degree and singular type of $V$ (see {\cite{Sw}}). 
Hence, as a next target, it seems reasonable to consider Problem \ref{prob} for the case in which $V$ is a normal del Pezzo surface with singularities worse than canonical singularities. 
In this article, we will give a partial answer to this problem by classifying normal del Pezzo surfaces of rank one with at most log canonical singularities containing the affine plane as a special kind of cylinder. 
In other words, we shall classify compactifications of the affine plane into normal del Pezzo surfaces of rank one defined over an algebraically non-closed field with at most log canonical singularities. 
We review compactifications of $\bA ^2$ in the next subsection. 
%%%%%%%%%%%%%%%%%%%%%%%%%%%%%%%%%%%%%%%%%%%%%%%%%%%%%%%%
\subsection{Compactifications of the affine plane}
%%%%%%%%%%%%%%%%%
Let $V$ be a normal projective surface over $\bk$ and let $D$ be a reduced effective divisor on $V$. 
Then we say that the pair $(V,D)$ is a {\it compactification} of the affine plane $\bA ^2_{\bk}$ if $V \backslash \Supp (D) \simeq \bA ^2_{\bk}$, moreover, $D$ is called the {\it boundary divisor}. 
It is known that the pair $(V,D)$ is a compactification of $\bA ^2_{\bk}$ if and only if $(V_{\kc},D_{\kc})$ is a compactification of $\bA ^2_{\kc}$ because of the fact that there is no non-trivial $\bk$-form of $\bA_{\kc}^2$ (see \cite{Ka75}). 
Now, assume that $V$ is of rank one, $V_{\kc}$ has at most log canonical singularities and $(V,D)$ is a compactification of $\bA ^2_{\bk}$. 
Kojima and Takahashi showed that $V$ is a numerical del Pezzo surface with at most rational singularities when $\bk = \kc$ (see {\cite{KT09}}). 
Notice that their argument works verbatim even if $\bk \not= \kc$. 
Furthermore, when $\bk$ is algebraically closed, the weighted dual graph of the boundary divisor $D$ is classified, and the list is summarized in {\cite{MZ88,Ko01,KT09}}. 
The goal of the article lies in a generalization of their works to the case where the base field is of characteristic zero without assuming that it is algebraically closed. 
%%%%%%%%%%%%%%%%%%%%%%%%%%%%%%%%%%%%%%%%%%%%%%%%%%%%%%%%
\subsection{Main results}\label{1-3}
%%%%%%%%%%%%%%%%%
In order to state our main result, we need to clarify the meaning of singular del Pezzo surfaces defined over an algebraically non-closed field. 
Let $S$ be a normal projective surface over $\bk$ and let $S_{\kc}$ be the base extension of $S$ to the algebraic closure $\kc$, i.e., $S_{\kc} = S \times _{\Spec (\bk )} \Spec (\kc )$. 
Then $S$ is called a {\it numerical del Pezzo surface} over $\bk$ if $-K_{S_{\kc}}$ is numerically ample, i.e., $(-K_{S_{\kc}})^2>0$ and $(-K_{S_{\kc}} \cdot C)>0$ for any closed curve $C$ on $S_{\kc}$, where these intersection numbers are in the sense of {\it Mumford's rational intersection number} (see, e.g., {\cite{Mu61,Sk84}}) since $S_{\kc}$ is not necessarily $\bQ$-factorial. 
However, note that any numerical del Pezzo surface $S_{\kc}$ is $\bQ$-factorial if $S$ is of rank one and contains the affine plane $\bA ^2_{\bk}$ (cf.\ {\cite[Theorem 1.1(2)]{KT09}}), in particular, we will mainly treat such numerical del Pezzo surfaces. 
We say that a numerical del Pezzo surface $S$ over $\bk$ is an {\it lc} (resp. a {\it log}, a {\it Du Val}) {\it del Pezzo surface} if the base extension $S_{\kc}$ has at most log canonical singularities (resp. quotient singularities, Du Val singularities). 
Finally, a numerical del Pezzo surface is said to have rank one if the rank of its Neron-Severi group equals one. 
%%%%%%%%%%%%%%%%%

Then our first main result states as follows (for some definitions, see Conventions and Notation): 
%%%%%%%%%%%%%%%%%
\begin{thm}\label{main}
Let $(S,\Delta )$ be a compactification of the affine plane over $\bk$. 
Assume that $S$ is a numerical del Pezzo surface of rank one over $\bk$ such that $\Sing (S_{\kc}) \not= \emptyset$. 
Let $\sigma :\wS \to S$ be the minimal resolution over $\bk$ and let $\wDelta$ be the reduced effective divisor on $\wS$ defined by $\wDelta := \sigma ^{\ast}(\Delta )_{\red}$. 
Then we have the following: 
\begin{enumerate}
\item $\wDelta _{\kc}$ is an SNC-divisor. 
\item The following three assertions about the number of singularities hold: 
\begin{enumerate}
\renewcommand{\labelenumii}{(\roman{enumii})}
\item $\sharp {\rm Sing}(S) \ge 1$. In other words, $S_{\kc}$ has a singular point, which is $\bk$-rational. 
\item $\sharp {\rm Sing}(S) \le 2$. Moreover, $\sharp {\rm Sing}(S)=2$ if and only if $\sharp {\rm Sing}(S_{\kc}) = 2$. 
\item $\sharp {\rm Sing}(S_{\kc}) = 1$ or $\rho _{\kc}(S_{\kc})+1$. 
\end{enumerate}
\item Assume further that $S$ is an lc del Pezzo surface. If $\rho _{\kc}(S_{\kc}) >1$, the weighted dual graph of $\wDelta _{\kc}$ is one of the graphs {\rm (1)}--{\rm (52)} in Appendix \ref{9}. 
\end{enumerate}
\end{thm}
%%%%%%%%%%%%%%%%%
\begin{rem}
In Theorem \ref{main}(3), if $\rho _{\kc}(S_{\kc})=1$, then the weighted dual graph of $\wDelta _{\kc}$ is classified (see {\cite[Appendix C]{Ko01}} and {\cite[Fig. 1]{KT09}}). 
\end{rem}
%%%%%%%%%%%%%%%%%
Furthermore, we obtain the method of determining whether any lc del Pezzo surface of rank one contains the affine plane or not. 
In other words, our next main result is the following: 
%%%%%%%%%%%%%%%%%
\begin{thm}\label{main(2)}
Let $S$ be an lc del Pezzo surface of rank one defined over $\bk$ such that $\Sing (S_{\kc}) \not= \emptyset$. 
Then the following two assertions are equivalents: 
\begin{enumerate}
\renewcommand{\labelenumi}{(\alph{enumi})}
\item $S$ contains the affine plane $\bA ^2_{\bk}$. 
\item Letting $\sigma :\wS \to S$ be the minimal resolution over $\bk$, there exists a reduced effective divisor $\wDelta$ on $\wS$ such that the exceptional set of $\sigma$ is included in $\Supp (\wDelta )$, each irreducible component of $\wDelta _{\kc}$ is a rational curve  and the weighted dual graph of $\wDelta _{\kc}$ is one of the graphs in {\cite[Appendix C]{Ko01}}, {\cite[Fig. 1]{KT09}} or Appendix \ref{9}. 
\end{enumerate}
\end{thm}
%%%%%%%%%%%%%%%%%
Moreover, as an application of Theorems \ref{main} and \ref{main(2)}, we obtain as follows: 
%%%%%%%%%%%%%%%%%
\begin{thm}\label{main(3)}
Let $S$ be a Du Val del Pezzo surface of rank one defined over $\bk$ such that $\Sing (S_{\kc}) \not= \emptyset$, and let $d$ be the degree of $S$, i.e., $d = (-K_S)^2 \in \{1,\dots ,6,8\}$. 
Then: 
\begin{enumerate}
\item If $d=8$, then $S$ contains $\bA ^2_{\bk}$ if and only if $S$ has a smooth $\bk$-rational point\footnote{Note that it is equivalent that $S$ is a trivial $\bk$-form of the weighted projective space $\bP _{\kc}(1,1,2)$.}. 
\item If $d=5$ or $6$, then $S$ always contains $\bA ^2_{\bk}$. 
\item If $d \le 4$, $S$ contains $\bA ^2_{\bk}$ if and only if the pair of the degree $d$ and the singularity type of $S_{\kc}$ is one of the following: 
\begin{align*}
(4,D_5),\,(4,D_4),\,(4,A_2+2A_1),\,(4,A_2),\,(3,E_6),\,(3,D_4), \, (2,E_7),\,(2,E_6),\,(2,A_6),\,(1,E_8). 
\end{align*}
\end{enumerate}
\end{thm}
%%%%%%%%%%%%%%%%%%%%%%%%%%%%%%%%%%%%%%%%%%%%%%%%%%%%%%%%
\subsection{Organization of article}

In \S \ref{2}, we shall review some basic properties of compactifications of the affine plane. 
In \S \ref{3}, we shall present some facts about twigs, where a twig is a special kind of weighted dual graph, which will play an important role in proving Theorem \ref{main}. 
In \S \ref{4}, we shall prepare some results of compactifications of the affine plane. 
These results will also play an important role in proving Theorem \ref{main}. 
In \S \ref{5}, we will prove Theorem \ref{main}(1) and (2), which are the key results in the next section. 
In \S \ref{6}, we will prove Theorem \ref{main}(3), in other words, we will classify compactifications of the affine plane into singular lc del Pezzo surfaces of rank one defined over an algebraically non-closed field. 
In \S \ref{7}, as an application of Theorem \ref{main}, we shall show Theorems \ref{main(2)} and \ref{main(3)}. Moreover, we will yield a criterion for del Pezzo fibrations with canonical singularities to contain vertical $\bA^2$-cylinders in terms of degree and singular type of generic fibers (see Corollary \ref{cano}). 
In \S \ref{8}, we will give various remarks about Theorem \ref{main}. 
Finally, Appendix \ref{9} summarizes the list of configurations of all compactifications of the affine plane into singular lc del Pezzo surfaces of rank one defined over an algebraically non-closed field in terms of weighted dual graphs. 
%%%%%%%%%%%%%%%%%%%%%%%%%%%%%%%%%%%%%%%%%%%%%%%%%%%%%%%%
\begin{conventions}
Let $X$ be an algebraic variety defined over $\bk$. 
Then $X_{\kc}$ denotes the base extension of $X$ to the algebraic closure $\kc$. 
Furthermore, we write $\Sing (X) := \Sing (X_{\kc}) \cap X(\bk )$. 

Let $D$ be a reduced effective divisor on a variety defined over $\bk$. 
Then $D_{\kc}$ denotes the base extension of $D$ to the algebraic closure $\kc$. 
We say that $D_{\kc}$ is an {\it SNC-divisor} if $D_{\kc}$ has only simple normal crossings. 
Moreover, $\sharp D$ denotes the number of all irreducible components in $\Supp (D)$ over $\bk$. 
Note that if $\Supp (D)$ contains an irreducible component, which is not geometrically irreducible, then the formula $\sharp D < \sharp D_{\kc}$ holds. 

For any weighted dual graph, a vertex $\circ$ with the number $m$ corresponds to an $m$-curve (see also the following Notation). 
Exceptionally, we omit this weight (resp. we omit this weight and use the vertex $\bullet$ instead of $\circ$) if $m=-2$ (resp. $m=-1$). 
\end{conventions}
%%%%%%%%%%%%%%%%%%%%%%%%%%%%%%%%%%%%%%%%%%%%%%%%%%%%%%%%
\begin{notation}
We will use the following notations: 
\begin{itemize}
\item $\rho _{\bk}(X)$: the rank of the Neron-Severi group of a variety $X$ defined over $\bk$. 
\item $\Cl (X)$: the divisor class group of a variety $X$. 
\item $\Pic (X)$: the Picard group of a variety $X$. 
\item $\varphi ^{\ast}(D)$: the total transform of a divisor $D$ by a morphism $\varphi$. 
\item $\psi _{\ast}(D)$: the direct image of a divisor $D$ by a morphism $\psi$. 
\item $(D \cdot D')$: the intersection number of two divisors $D$ and $D'$. 
\item $(D)^2$: the self-intersection number of a divisor $D$. 
\item $\bF _m$: the Hirzebruch surface of degree $m$, i.e., $\bF _m = \bP (\sO _{\bP ^1} \oplus \sO _{\bP ^1}(m))$. 
\item $m$-curve: a smooth projective rational curve defined over $\kc$ with self-intersection number $m$. 
\end{itemize}
\end{notation}
%%%%%%%%%%%%%%%%%%%%%%%%%%%%%%%%%%%%%%%%%%%%%%%%%%%%%%%%
\begin{acknowledgment}
The author would like to thank his supervisor Professor Takashi Kishimoto for his encouragement, useful comments and many discussions. 
Also, he is grateful to the referee for suggesting many valuable comments that helped to improve this article. 
\end{acknowledgment}
%%%%%%%%%%%%%%%%%%%%%%%%%%%%%%%%%%%%%%%%%%%%%%%%%%%%%%%%%%%%%%%%%%%%%%%%%%%%%%%%%%%%%%%%%%%%%%%%%%%%%%%%%%
\section{Preliminaries}\label{2}
%%%%%%%%%%%%%%%%%
In this section, we shall recall some basic facts about compactifications of the affine plane. 
%%%%%%%%%%%%%%%%%%%%%%%%%%%%%%%%%%%%%%%%%%%%%%%%%%%%%%%%
\subsection{Normal projective surfaces containing the affine plane}\label{2-1}
%%%%%%%%%%%%%%%%%
\begin{lem}\label{lem(2-1)}
For any compactification $(V,D)$ of the affine plane over $\bk$, the boundary divisor $D$ has no cycle. 
\end{lem}
%%%%%%%%%%%%%%%%%
\begin{proof}
This fact seems well-known to the experts. See, e.g., {\cite[Lemma 2.12]{Sw}}. 
\end{proof}
%%%%%%%%%%%%%%%%%
\begin{lem}\label{lem(2-2)}
For any compactification $(V,D)$ of the affine plane over $\bk$, we have $\sharp D = \rho _{\bk}(V)$. 
\end{lem}
%%%%%%%%%%%%%%%%%
\begin{proof}
Let us put $n:= \sharp D$ and $U := V \backslash \Supp (D)$. 
Let $C_1,\dots ,C_n$ be all irreducible components of $D$, let us take a free abelian group $\bigoplus _{i=1}^n\bZ C_i$ generated by $C_1,\dots ,C_n$, and let $f:\bigoplus _{i=1}^n\bZ C_i \to \Cl (V)$ be the group homomorphism defined by $f(C_i) := [C_i]$. 
Then $f$ is surjective by virtue of $\Cl (U) \simeq \Pic (U) = 0$ (cf.\ {\cite[Lemma 4.6]{KPZ14}}). 
Moreover, we know that $f$ is injective since the coordinate ring $R$ of $U$ satisfies $R ^{\times}=\bk ^{\times}$. 
Thus, we have a group isomorphism $\bigoplus _{i=1}^n\bZ C_i \simeq \Cl (V)$. 
In particular, $n = \rho _{\bk}(V)$. 
\end{proof}
%%%%%%%%%%%%%%%%%
\begin{lem}\label{lem(2-3)}
Assume $\bk = \kc$. 
Let $V$ be a normal projective surface of rank one with at most log canonical singularities over $\bk$ containing the affine plane $\bA ^2_{\bk}$. 
Then $\sharp \Sing (V) \le 2$. 
\end{lem}
%%%%%%%%%%%%%%%%%
\begin{proof}
See {\cite[Lemma 4.2]{KT09}}. 
\end{proof}
%%%%%%%%%%%%%%%%%%%%%%%%%%%%%%%%%%%%%%%%%%%%%%%%%%%%%%%%
\subsection{Minimal normal compactifications of the affine plane}\label{2-2}
%%%%%%%%%%%%%%%%%
In this subsection, assume that $\bk$ is algebraically closed. 
We recall minimal normal compactifications of the affine plane defined over an algebraically closed field of characteristic zero. 
We refer to {\cite{Mo73,Ki02}}. 
%%%%%%%%%%%%%%%%%
\begin{defin}
Let $(V,D)$ be a compactification of the affine plane $\bA ^2_{\bk}$ over $\bk$ such that $V$ is smooth. 
Then this pair $(V,D)$ is a {\it minimal normal compactification} (an {\it mnc}, for short) of $\bA ^2_{\bk}$ if $D$ is an SNC-divisor and any $(-1)$-curve $E$ in $\Supp (D)$ satisfies $(E \cdot D-E) \ge 3$. 
\end{defin}
%%%%%%%%%%%%%%%%%
Morrow classified the minimal normal compactifications of the affine plane when the base field is $\bC$ ({\cite{Mo73}}). 
In this article, we will mainly use the following facts: 
%%%%%%%%%%%%%%%%%
\begin{lem}\label{Morrow}
Let $(V,D)$ be an mnc of $\bA ^2_{\bk}$ over $\bk$. 
Then the following assertions hold: 
\begin{enumerate}
\item Any irreducible component of $D$ is a smooth rational curve and the dual graph of $D$ is a linear chain. In particular, $\Supp (D)$ does not contain any $(-1)$-curve on $V$. 
\item If $\sharp D = 1$, then $(D)^2=1$. 
\item If $\sharp D = 2$, then $D$ contains at least one irreducible component, say $\Gamma$, satisfying $(\Gamma )^2=0$. 
\item If $\sharp D \ge 3$, then $D$ contains exactly two irreducible components, say $\Gamma _0$ and $\Gamma _+$, satisfying $(\Gamma _0)^2=0$, $(\Gamma_+)^2>0$ and $(\Gamma _0 \cdot \Gamma _+)=1$. 
\end{enumerate}
\end{lem}
%%%%%%%%%%%%%%%%%
\begin{proof}
See {\cite{Mo73}}. 
\end{proof}
%%%%%%%%%%%%%%%%%
Furthermore, it is known that the converse of Morrow's result is true. 
More precisely, the following result holds: 
%%%%%%%%%%%%%%%%%
\begin{lem}\label{Kishimoto}
Let $V$ be a smooth projective surface defined over $\bk$ and let $D$ be a reduced divisor on $V$ such that $V \backslash \Supp (D)$ is affine and each irreducible component of $D$ is a rational curve. 
If the weighted dual graph of $D$ is the same as that of the boundary divisor of an mnc of the affine plane $\bA ^2_{\bk}$, then $(V,D)$ is an mnc of $\bA ^2_{\bk}$. 
\end{lem}
%%%%%%%%%%%%%%%%%
\begin{proof}
See {\cite{Ki02}}. 
\end{proof}
%%%%%%%%%%%%%%%%%%%%%%%%%%%%%%%%%%%%%%%%%%%%%%%%%%%%%%%%%%%%%%%%%%%%%%%%%%%%%%%%%%%%%%%%%%%%%%%%%%%%%%%%%%
\section{Properties of twigs}\label{3}
%%%%%%%%%%%%%%%%%%%%%%%%%%%%%%%%%%%%%%%%%%%%%%%%%%%%%%%%
\subsection{Some definitions}\label{3-1}
%%%%%%%%%%%%%%%%%
Let $D$ be an SNC-divisor on a smooth projective surface defined over $\kc$. 
Let $A$ be the weighted dual graph of $D$. 
If $A$ is given by the following graph: 
%%%%%%%%%%%%%%%%%
\begin{align*}
\xygraph{
\circ ([]!{+(0,-.3)} {^{-m_1}}) -[r] \circ ([]!{+(0,-.3)} {^{-m_2}}) -[r] \cdots ([]!{+(0,-.3)} {}) -[r] \circ ([]!{+(0,-.3)} {^{-m_r}})
},
\end{align*}
%%%%%%%%%%%%%%%%%
then $A$ is called the {\it twig} and we write this weighted dual graph $[m_1,\dots ,m_r]$ as $A$. 
In this subsection, we will present some definitions for the twig. 
The following definition is based on {\cite{Fuj82}} (see also {\cite{Ko01}}): 
%%%%%%%%%%%%%%%%%
\begin{defin}\label{def(3-1)}
Let $A =[m_1,\dots ,m_r]$ be a twig. 
Then the twig $[m_r, m_{r-1}, \ldots , m_1]$ is called the {\it transposal} of $A$ and denoted by ${}^tA$. 
We define also $\overline{A} := [m_2, \dots, m_r]$ and $\underline{A} := [m_1 ,\dots ,m_{r-1}]$, where we put $\overline{A} = \underline{A} = \emptyset$ if $r=1$. 
We say that $A$ is {\it admissible} if $m_i \ge 2$ for any $i=1,\dots ,r$. 
In what follows, we assume that $A$ is admissible. 
Then $d(A)$ denotes the absolute value of the determinant of the intersection matrix corresponding to $A$ and is simply called the {\it determinant} of $A$, where we put $d(\emptyset )=1$. 
We say that $e(A) := d(\overline{A})/d(A )$ is the {\it inductance} of $A$. 
By {\cite[Corollary (3.8)]{Fuj82}}, $e$ defines a one-to-one correspondence from the set of all admissible twigs to the set of rational numbers in the interval $(0,1)$. 
Hence, there exists uniquely an admissible twig ${A}^*$, whose inductance is equal to $1-e({}^tA )$, so that we call the admissible twig ${A}^*$ the {\it adjoint} of $A$.
\end{defin}
%%%%%%%%%%%%%%%%%
\begin{eg}
Consider two admissible twigs $A := [2,4]$ and $B := [2,2,3]$. 
Then $d({^tA})=d(A)=7$ and $d(\overline{{}^tA})=2$, so that $e({^tA}) = \frac{2}{7}$. 
Moreover, $d(B)=7$ and $d(\overline{B})=5$, namely, $e(B) = \frac{5}{7} = 1-e({}^tA)$. 
Hence, $B = A^*$. 
\end{eg}
%%%%%%%%%%%%%%%%%
Furthermore, we will use the following notation in this article: 
%%%%%%%%%%%%%%%%%
\begin{defin}
\begin{enumerate}
\item Letting ${A}_1,\dots ,{A}_s$ be twigs given by ${A}_i = [m_{i,1},\dots ,m_{i,r_i}]$ for $i=1,\dots ,s$, we write $[{A}_1,\dots ,{A}_s] := [m_{1,1},\dots ,m_{1,r_1},\dots \dots ,m_{s,1},\dots ,m_{s,r_s}]$. 
\item For a positive integer $t$, we write $[t \times 2] := [\underbrace{2,\dots ,2}_{t\text{-times}}]$. 
\item For a positive integer $m$ and a non-negative integer $t$, we write two twigs $L(m;t)$ and $R (m;t)$ respectively as follows: 
\begin{align*}
L (m;t) :=
\left\{ \begin{array}{ll}
{[} [t \times 2],[m]{]} &\text{if}\ t>0 \\
{[}m{]} & \text{if}\ t=0 
\end{array} \right. , \quad 
R (m;t) :=
\left\{ \begin{array}{ll}
{[}[m],[t \times 2]{]} & \text{if}\ t>0 \\
{[}m{]} & \text{if}\ t=0 
\end{array} \right. .
\end{align*}
\end{enumerate}
\end{defin}
%%%%%%%%%%%%%%%%%%%%%%%%%%%%%%%%%%%%%%%%%%%%%%%%%%%%%%%%
\subsection{Twigs contracted to single smooth rational curves}\label{3-2}
%%%%%%%%%%%%%%%%%
Let $D$ be an SNC-divisor on a smooth projective surface over $\kc$ such that any irreducible component of $D$ is a rational curve. 
We recall the following result for later use in \S \S \ref{3-3}. 
%%%%%%%%%%%%%%%%%
\begin{lem}\label{Fujita}
Let $D$ be the same as above. Assume that the weighted dual graph of $D$ is the twig $[A ,[1],B ]$ for some admissible twigs $A$ and $B$. 
Then $D$ can be contracted to a $0$-curve if and only if $B ={A}^*$. 
\end{lem}
%%%%%%%%%%%%%%%%%
\begin{proof}
See, {\cite[Proposition (4.7)]{Fuj82}}. 
\end{proof}
%%%%%%%%%%%%%%%%%
In Lemma \ref{Fujita}, since the adjoint of any admissible twig is unique, note that $B$ is uniquely determined according to $A$. 
By applying Lemma \ref{Fujita}, we obtain the following result: 
%%%%%%%%%%%%%%%%%
\begin{lem}\label{Fujita'}
Let $D$ be the same as above. Assume that the weighted dual graph of $D$ is the twig $[[m],A ,[1],B ]$ for some integer $m$ and admissible twigs $A$ and $B$. 
Then $D$ can be contracted to the twig $[m,1]$ if and only if $B =\underline{{A}^*}$. 
\end{lem}
%%%%%%%%%%%%%%%%%
\begin{proof}
Since $A$ is admissible, notice that $A$ can be uniquely denoted by $[L (m_r;t_r),\dots ,L (m_1;t_1)]$ for some $m_1 \ge 2$, $m_i \ge 3$ $(i>1)$ and $t_j \ge 0$ $(1 \le j \le r)$. 
By the induction on $r$, we see that $D$ can be contracted to the twig $[m,1]$ if and only if $B$ can be written as follows:  
\begin{align}\label{(3.1)}
B = \left\{ \begin{array}{ll} {[(m_1-2) \times 2]} & \text{\ if\ }r=1 \\ {[[(m_1-2) \times 2],R (t_1+3;m_2-3),\dots ,R (t_{r-1}+3;m_r-3)]} &  \text{\ if\ }r>1 \\ \end{array} \right. .
\end{align}
Let $D'$ be an SNC-divisor on a smooth projective surface over $\kc$ such that the weighted dual graph of $D'$ is the twig $[A ,[1],{B}']$, where ${B}'$ is the admissible twig defined by: 
\begin{align*}
{B}' := \left\{ \begin{array}{ll} {[}[(m_1-2)\times 2],[t_1-2]{]} & \text{\ if\ } r=1 \\ 
{[}[(m_1-2)\times 2],R (t_1+3;m_2-3),\dots ,R (t_{r-1};m_r-3),[t_r-2]{]} & \text{\ if\ } r>1 \\ \end{array} \right. .
\end{align*}
Then $D'$ can be contracted to a $0$-curve by induction on $r$. 
Hence, we have ${B}'={A}^*$ by Lemma \ref{Fujita}. 
In particular, $B$ is as in (\ref{(3.1)}) if and only if $B =\underline{{B}'}=\underline{{A}^*}$. 
\end{proof}
%%%%%%%%%%%%%%%%%
\begin{defin}\label{mA}
Let $A = [L (m_r;t_r),\dots ,L (m_1;t_1)]$ be an admissible twig, $m_1 \ge 2$, $m_i \ge 3$ $(i>1)$ and $t_j \ge 0$ $(1 \le j \le r)$. 
In this article, we then put $m_{A} := t_r+3$. 
\end{defin}
%%%%%%%%%%%%%%%%%
\begin{rem}\label{rem(mA)}
Let $A$ be an admissible twig. 
By definition of $m_A$ and Lemma \ref{Fujita'}, the twig $[[m],A ,[1],\underline{{A}^*},[m_{A}]]$ can be contracted to the twig $[m,1,2]$ for an arbitrary integer $m$. 
\end{rem}
%%%%%%%%%%%%%%%%%
\begin{eg}
Consider the admissible twig $A := [2,4]$. 
By definition, we know $m_A=4$. 
Meanwhile, since $A^* = [2,2,3]$, we obtain $\underline{{A}^*} = [2,2]$. 
For an arbitrary integer $m$, the twig $[[m],A ,[1],\underline{{A}^*},[m_{A}]] = [m,2,4,1,2,2,4]$ is then contracted to $[m,1,2]$ as follows: 
\begin{align*}
[m,2,4,1,2,2,4] &\to [m,2,3,1,2,4] \to [m,2,2,1,4] \to [m,2,1,3] \to [m,1,2]. 
\end{align*}
\end{eg}
%%%%%%%%%%%%%%%%%%%%%%%%%%%%%%%%%%%%%%%%%%%%%%%%%%%%%%%%
\subsection{Twigs as boundary divisors of the affine plane}\label{3-3}
%%%%%%%%%%%%%%%%%
In this subsection, let $(\wV ,\wD )$ be a compactification of the affine plane $\bA ^2_{\kc}$ such that $\wV$ is a smooth projective surface over $\kc$ and the weighted dual graph of $\wD$ is the twig $[m_1,\dots ,m_r]$ with $m_i \ge 1$ for any $i=1,\dots ,r$. 
By Lemma \ref{Morrow}, $\wD$ consists of irreducible components $\{ \wC _i \} _{1 \le i \le r}$ such that $\wC _i$ is a $(-m_i)$-curve for $i=1,\dots ,r$, moreover, we see that $(\wV ,\wD )$ is not an mnc of $\bA ^2_{\kc}$. 
Let 
$\nu :\wV \to \cV$ be a sequence of contractions of $(-1)$-curves and subsequently (smoothly) contractible curves in $\Supp (\wD )$ such that the pair $(\cV ,\cD )$ is an mnc of $\bA ^2_{\kc}$, where $\cD := \nu _{\ast}(\wD )$. 
%%%%%%%%%%%%%%%%%
\begin{lem}\label{lem(3-1)}
With the notation as above, then the following three assertions hold: 
\begin{enumerate}
\item $r \ge 3$. 
\item There exists at least one integer $e$ with $2 \le e \le r-1$ such that $m_e=1$. 
\item If $r=3$, the weighted dual graph of $\wD$ is the twig $[1,1,m]$ or $[m,1,1]$ for some $m \ge 1$. 
\end{enumerate}
\end{lem}
%%%%%%%%%%%%%%%%%
\begin{proof}
In (1), supposing $n \le 2$, we can easily obtain a contradiction by Lemma \ref{Morrow}. 

In (2), noticing $r \ge 3$ by (1), suppose $m_i \ge 2$ for any $i=2,\dots ,r-1$. 
Since $(\wV ,\wD )$ is not an mnc of $\bA ^2_{\kc}$, we obtain $m_1=1$ or $m_r=1$. 
Then $\wD$ can be contracted to the twig $[m]$ for some non-negative integer $m$ or an admissible twig. 
It contradicts Lemma \ref{Morrow}. 

In (3), we note $m_2=1$ by (2). 
Hence, $\wD$ can be contacted to the twig $[m_1-1,m_3-1]$. 
By Lemma \ref{Morrow}(2) and (3), we see $m_1-1=0$ or $m_3-1=0$. 
This completes the proof. 
\end{proof}
%%%%%%%%%%%%%%%%%
\begin{lem}\label{lem(3-2)}
With the notation as above, then the following two assertions hold: 
\begin{enumerate}
\item Assume that there exists exactly one integer $e$ with $1\le e \le r-1$ such that $m_i=1$ if and only if $i=e$ or $e+1$. 
Then we obtain $r=3$. 
\item Assume that $r \ge 4$, $m_i=m_{r+1-i}$ for any $i$ and there exists exactly one integer $e$ with $1 \le e < \frac{r}{2}$ such that $m_i=1$ if and only if $i=e$, $e+1$, $r-e$ or $r+1-e$. 
Then we obtain $r=4$. Namely, the weighted dual graph of $\wD$ is $[1,1,1,1]$. 
\end{enumerate}
\end{lem}
%%%%%%%%%%%%%%%%%
\begin{proof}
In (1), notice that $r \ge 3$ by Lemma \ref{lem(3-1)}(1). 
We may assume $e+1<r$ by symmetry. 
Moreover, we can assume that $\nu$ starts with the contraction of $\wC _e$. 
Then we see that $\sharp \cD \ge 2$, and any irreducible component of $\cD$ with self-intersection number $\ge 0$ is only $\cC := \nu _{\ast}(\wC _{e+1})$. 
Thus, we have $\sharp \cD = 2$ and $({\cC})^2=0$ by Lemma \ref{Morrow}. 
This means $r=3$ by virtue of $\sharp \wD - \sharp \cD = (\cC )^2-(\wC _{e+1})^2 =1$. 

In (2), we can assume that $\nu$ starts with the contraction of $\wC _e+\wC_{r+1-e}$. 
Then we see that $\sharp \cD \ge 2$, and any irreducible component of $\cD$ with self-intersection number $\ge 0$ is only $\cC _1 :=\nu _{\ast}(\wC _{e+1})$ or $\cC _2 :=\nu _{\ast}(\wC _{r-e})$. 
Moreover, by $m_{e+1}=m_{r-e}$ and construction of $\nu$, we obtain $(\cC _1)^2=(\cC _2)^2$. 
Thus, we have $\sharp \cD = 2$ and $(\cC _1)^2=(\cC _2)^2=0$ by Lemma \ref{Morrow}. 
This means $r=4$ by virtue of $\sharp \wD - \sharp \cD = (\cC _1)^2-(\wC _{e+1})^2+(\cC _2)^2-(\wC _{r-e})^2 =2$. 
\end{proof}
%%%%%%%%%%%%%%%%%
\begin{lem}\label{lem(3-4)}
With the notation as above, assume that $r = \sharp D$ is odd, $m_i=m_{r+1-i}$ for any $i$, and there exists exactly one integer $e$ with $1 \le e \le r'$ such that $m_i=1$ if and only if $i=e$ or $r+1-e$, where $r' := \frac{r+1}{2}$. 
Then the following assertions hold: 
\begin{enumerate}
\item $e \not= r'$.  
\item If $e = r'-1$, then we obtain $m_{r'}=r-2$ and $m_i=2$ for any $i=1,\dots ,r$ with $|r'-i|>1$, namely, the weighted dual graph of $\wD$ can be denoted by $[L (1;r'-2),[r-2],R (1;r'-2)]$. 
\end{enumerate}
\end{lem}
%%%%%%%%%%%%%%%%%
\begin{proof}
In (1), suppose $e=r'$. 
Noting $r \ge 5$ by Lemma \ref{lem(3-1)}, the weighted dual graph of the contraction of $\wC _{r'}$ is the twig $[m_1,\dots ,m_{r'-2},m_{r'-1}-1,m_{r'+1}-1,m_{r'+2},\dots ,m_r]$, which is not admissible by Lemma \ref{Morrow}. 
Hence, $m_{r' \pm 1}-1=1$, which contradicts Lemma \ref{lem(3-2)}(1). 

In (2), it can be shown by the induction on $r$ combined with Lemmas \ref{Morrow} and \ref{lem(3-1)}. 
Hence, it is left to the reader. 
\end{proof}
%%%%%%%%%%%%%%%%%
Finally, we prepare the following proposition, which will play an important role in \S \ref{6}: 
%%%%%%%%%%%%%%%%%
\begin{prop}\label{prop(3)}
With the notation as above, then the following assertions hold: 
\begin{enumerate}
\item Assume that there exists exactly one integer $e$ satisfying $m_e=1$. 
Then the weighted dual graph of $\wD$ can be denoted by $[A ,[1],{A}^*,[m]]$ for some admissible twig $A$. 
\item Assume that $r = \sharp \wD$ is even, $m_i=m_{r+1-i}$ for any $i$, and there exists exactly one integer $e$ with $1 \le e \le \frac{r}{2}$ such that $m_i=1$ if and only if $i=e$ or $r+1-e$. 
Then the weighted dual graph of $\wD$ can be denoted by $[{}^t({A}^*),[1],{}^tA ,A ,[1],{A}^*]$ for some admissible twig $A$. 
\item Assume that $r = \sharp \wD$ is odd, $m_i=m_{r+1-i}$ for any $i$, and there exists exactly one integer $e$ with $1 \le e \le r'$ such that $m_i=1$ if and only if $i=e$ or $r+1-e$, where $r' := \frac{r+1}{2}$. 
Then the weighted dual graph of $\wD$ can be denoted by one of the following: 
\begin{itemize}
\item $[L (1;r'-2),[r-2],R (1;r'-2)]$. 
\item $[L (m_{A};t),{}^t(\underline{{A}^*}),[1],{}^t{A},[2t+3],A ,[1],\underline{{A}^*},R(m_{A};t)]$ for some admissible twig $A$, where $m_{A}$ is as in Definition \ref{mA} and $t$ is a non-negative integer. 
\end{itemize} 
\end{enumerate}
\end{prop}
%%%%%%%%%%%%%%%%%
\begin{proof}
Each assertion can be shown by the induction on $r$ combined with Lemmas \ref{Morrow}, \ref{Fujita}, \ref{Fujita'}, \ref{lem(3-1)}, \ref{lem(3-2)} and \ref{lem(3-4)}. 
However, the proof of (3) is a bit intricate. 
Hence, we only treat (3), and the proofs of (1) and (2) are left to the reader. 
Notice $r \ge 5$ by Lemma \ref{lem(3-1)}(3). 
Assume $r=5$. Then we can easily see $[m_1,\dots ,m_5]=[2,1,3,1,2]$ by using Lemmas \ref{lem(3-1)} and \ref{lem(3-4)}. 
Assume $r=7$. Then we can easily see either $[m_1,\dots ,m_7]=[2,2,1,5,1,2,2]$ or $[m_1,\dots ,m_7]=[3,1,2,3,2,1,3]$ by using Lemmas \ref{lem(3-1)}, \ref{lem(3-2)} and \ref{lem(3-4)}, where we note $\underline{[2]^*}=\emptyset$ and $m_{[2]}=3$ because of $[2] = L(2;0)$. 
Assume $r>7$. We put $e':=r+1-e$ for simplicity. 
If $e=r'-1$, then it follows from Lemma \ref{lem(3-4)}(2). 
Thus, in what follows, we may assume $e<r'-1$. Namely, $e'-e>2$. 
Let ${B}'$ be the weighted dual graph of the contraction of $\wC _e + \wC _{e'}$ on $\wD$. 
By Lemmas \ref{Morrow} and \ref{lem(3-2)}(2), $B'$ is a non-admissible twig and further satisfies the hypothesis of (3). 
If ${B}'=[L (1;r'-3),[r-4],R (1;r'-3)]$, then $[m_1,\dots ,m_r] = [L (3;r'-4), [1],[2],[r-4],[2],[1],R (3;r'-4)]$ by the assumption $e<r'-1$, where we note $r-4=2(r'-4)+3$. 
In what follows, we assume that ${B}'=[{L}(m_{{A}'};t),{}^t(\underline{{{A}'}^*}),[1],{}^t{{A}'},[2t+3],{A}',[1],\underline{{{A}'}^*},R (m_{{A}'};t)]$ for some admissible twig ${A}'$, where $t$ is a non-negative integer. 
Hence, $[m_{r'},\dots ,m_{r-t-1}]$ can be contracted to $[[2t+3],{A}',[1],\underline{{{A}'}^*}]$. 
Since $[[2t+3],{A}',[1],\underline{{{A}'}^*}]$ can be contracted to $[2t+3,1]$ by Lemma \ref{Fujita'} (see also Remark \ref{mA}), so is $[m_{r'},\dots ,m_{r-t-1}]$. 
Hence, by using Lemma \ref{Fujita'} again, $[m_{r'},\dots ,m_{r-t-1}]=[[2t+3],A ,[1],\underline{{A}^*}]$, where $A := [m_{r'+1},\dots ,m_{e'-1}]$. 
Meanwhile, since $[m_{r'},\dots ,m_r]=[[2t+3],A ,[1],\underline{{A}^*},[m_{r-t},\dots ,m_r]]$ can be contracted to $[[2t+3,1],R (2;t)]$, we know $[m_{r-t},\dots ,m_r]=R (m_{A};t)$. 
By symmetry, we thus obtain $[m_1,\dots ,m_r]=[L (m_{A};t),{}^t(\underline{{A}^*}),[1],{}^tA ,[2t+3],A ,[1],\underline{{A}^*},R (m_{A};t)]$. 
\end{proof}
%%%%%%%%%%%%%%%%%
By Proposition \ref{prop(3)} combined with Lemmas \ref{Fujita} and \ref{Fujita'}, if $\wD$ satisfies the assumption of Proposition \ref{prop(3)}(1) (resp. (2), (3)), then we can take the birational morphism $\nu :\wV \to \cV$ with $\wV \backslash \Supp (\wD ) \simeq \cV \backslash \Supp (\cD ) \simeq \bA ^2_{\kc}$ such that $\cV \simeq \bF _m$ for some $m \ge 2$ (resp. $\cV \simeq \bP ^1_{\kc} \times \bP ^1_{\kc}$, $\cV \simeq \bP ^2_{\kc}$) and $\cD$ consists of the minimal resolution and a closed fiber of the structure morphism $\bF _m \to \bP ^1_{\kc}$ (resp. two fibers of first and second projection $\bP ^1 _{\kc} \times \bP ^1_{\kc} \to \bP ^1_{\kc}$, a projective line), where $\cD := \nu _{\ast}(\wD)$
%%%%%%%%%%%%%%%%%%%%%%%%%%%%%%%%%%%%%%%%%%%%%%%%%%%%%%%%%%%%%%%%%%%%%%%%%%%%%%%%%%%%%%%%%%%%%%%%%%%%%%%%%%
\section{Boundaries of compactifications of the affine plane}\label{4}
%%%%%%%%%%%%%%%%%
Let $(\wV ,\wD )$ be a compactification of the affine plane $\bA ^2_{\bk}$ over $\bk$ such that $\wV$ is a smooth projective surface over $\bk$ and $\wD _{\kc}$ is an SNC-divisor on $\wV _{\kc}$ and consists of only rational curves with self-intersection number $\le -1$. 
In this section, in order to prove Theorem \ref{main}, we study the boundary $\wD$ according to the configuration of $(-1)$-curves on $\Supp (\wD _{\kc})$. 
%%%%%%%%%%%%%%%%%

Arguing as in the proof of {\cite[Lemma 4.3]{KT09}}, we then obtain the following result: 
%%%%%%%%%%%%%%%%%
\begin{lem}\label{lem(4-3)}
Let the notation be the same as above, and let $\wE _0,\wE _1,\dots ,\wE _r$ be all $(-1)$-curves in $\Supp (\wD _{\kc})$. 
Assume that $r \ge 1$,  $(\wE _0 \cdot \wE _i)=1$ for $i=1,\dots ,r$ and the union $\sum _{i=1}^r\wE _i$ is disjoint. 
Then $(\wE _0 \cdot \wD -\wE _0) \le 2$. 
\end{lem}
%%%%%%%%%%%%%%%%%
\begin{proof}
Suppose that $(\wE _0 \cdot \wD -\wE _0) \ge 3$. 
We may assume that $(\wE _i \cdot \wD -\wE _i) \le 2$ for every $i=1,\dots ,r$ with $i \le r'$ and $(\wE _i \cdot \wD -\wE _i) \ge 3$ for every $i=1,\dots ,r$ with $i>r'$, where $r'$ is an integer between $0$ and $r$. 
Since $(\wV _{\kc},\wD _{\kc})$ is not an mnc of $\bA ^2_{\kc}$ by Lemma \ref{Morrow}(1), we notice $r'>0$. 
Let $\nu :\wV _{\kc} \to \cV$ be a sequence of contractions of $(-1)$-curves and subsequently (smoothly) contractible curves in $\Supp (\wD _{\kc})$ defined over $\kc$, starting with the contraction of the disjoint union $\sum _{i=1}^{r'}\wE _i$, such that the pair $(\cV ,\cD )$ is an mnc of $\bA ^2_{\kc}$, where $\cD := \nu _{\ast}(\wD _{\kc})$. 
Putting $\cE := \nu _{\ast}(\wE _0)$, we notice $\cE \not= 0$, moreover, $(\cE \cdot \cD - \cE ) \le 2$ and $(\cE )^2 \ge 0$ by Lemma \ref{Morrow}(1). 
By Lemma \ref{lem(2-1)}, any irreducible component of $\cD - \cE$ has self-intersection number $\le -1$ (if it exists at all). 
Thus, $\sharp \cD \le 2$ by Lemma \ref{Morrow}(4). 
Hence, $(\wE _i \cdot \wD - \wE _i) \le 2$ for $i=1,\dots ,r$. 
If $\sharp \cD = 2$, we obtain $(\cE )^2=0$ by Lemma \ref{Morrow}(3) and $r+1 \ge (\wE _0 \cdot \wD - \wE _0) \ge 3$ by the assumption. 
However, we then have $0 = (\cE )^2 \ge (\wE _0)^2+r \ge 1$, which is absurd. 
If $\sharp \cD = 1$, we obtain $(\cE )^2=1$ by Lemma \ref{Morrow}(2) and $r = (\wE _0 \cdot \wD - \wE _0) \ge 3$ by the assumption. 
However, we then have $1 = (\cE )^2 \ge (\wE _0)^2+r \ge 2$, which is absurd. 
\end{proof}
%%%%%%%%%%%%%%%%%
Next, we show the following two lemmas generalizing Lemma \ref{lem(3-2)}: 
%%%%%%%%%%%%%%%%%
\begin{lem}\label{lem(4-1)}
With the notation as above, assume further that $\wD _{\kc}$ contains exactly two $(-1)$-curves $\wE _1$ and $\wE _2$ and they satisfy $(\wE _1 \cdot \wE _2)=1$. 
Then the weighted dual graph of $\wD _{\kc}$ is the twig $[1,1,m]$ for some integer $m \ge 2$. 
\end{lem}
%%%%%%%%%%%%%%%%%
\begin{proof}
Let $\nu :\wV _{\kc} \to \cV$ be a sequence of contractions of $(-1)$-curves and subsequently (smoothly) contractible curves in $\Supp (\wD _{\kc})$ defined over $\kc$ such that the pair $(\cV ,\cD )$ is an mnc of $\bA ^2 _{\kc}$, where $\cD := \nu _{\ast}(\wD _{\kc})$. 
By Lemma \ref{lem(4-3)}, we notice $(\wE _i \cdot \wD - \wE _i) \le 2$ for $i=1,2$. 
In particular, $\nu \not= id$. 
Noting that $\wD _{\kc}$ is connected and has no cycle by Lemma \ref{lem(2-1)}, the divisor $\wD _{\kc} -(\wE _1+\wE _2)$ can be decomposed into connected components $\wD_ 1+\wD _2$ such that $(\wD _i \cdot \wE _j)=0$ for $i,j=1,2$ with $i \not= j$, where it is not necessarily $\wD _i \not= 0$ for $i=1,2$. 
Since we can assume that $\nu$ starts with the contraction of $\wE _1$ (resp. $\wE _2$), the weighted dual graph of $\nu _{\ast}(\wD _2+\wE _2)$ (resp. $\nu _{\ast}(\wD _1+\wE _1)$) is then a twig by Lemma \ref{Morrow}(1). 
In particular, the weighted dual graph of $\wD _{\kc}$ is a twig. 
Hence, this assertion follows from Lemma \ref{lem(3-2)}(1). 
\end{proof}
%%%%%%%%%%%%%%%%%
\begin{lem}\label{lem(4-2)}
With the notation as above, assume further that $\wD _{\kc}$ contains exactly four $(-1)$-curves $\wE _1$, $\wE _2$, $\wE _3$ and $\wE _4$, they satisfy $(\wE _1 \cdot \wE _2)=(\wE _3 \cdot \wE _4)=1$, and $\wE _1$ and $\wE _4$ lie in the same $\Gal$-orbit. 
Then the weighted dual graph of $\wD _{\kc}$ is the twig $[1,1,1,1]$. 
\end{lem}
%%%%%%%%%%%%%%%%%
\begin{proof}
Let $\nu :\wV _{\kc} \to \cV$ be a sequence of contractions of $(-1)$-curves and subsequently (smoothly) contractible curves in $\Supp (\wD _{\kc})$ defined over $\kc$ such that the pair $(\cV ,\cD )$ is an mnc of $\bA ^2 _{\kc}$, where $\cD := \nu _{\ast}(\wD _{\kc})$. 
Since $\Supp (\wD _{\kc})$ contains a $(-1)$-curve, we notice $\nu \not= id$ by Lemma \ref{Morrow}(1). 
Hence, we may assume $(\wE _1 \cdot \wD - \wE _1) \le 2$. 

Now, suppose $(\wE _1 \cdot \wE _4)=1$. 
Then $(\wE _2 \cdot \wE _3)=0$ by Lemma \ref{lem(2-1)}. 
If $(\wE _2 \cdot \wD -\wE _2) \ge 3$ (resp. $(\wE _3 \cdot \wD -\wE _3) \ge 3$), by contracting $\wE _1$ (resp. $\wE _4$) we obtain the mnc of $\bA ^2_{\kc}$ such that the direct image $\wE _2$ (resp. $\wE _3$) is a $(-1)$-curve and is included in it boundary. 
This contradicts Lemma \ref{Morrow}(1). 
Hence, we have $(\wE _2 \cdot \wD -\wE _2) \le 2$ and $(\wE _3 \cdot \wD -\wE _3) \le 2$. 

Hence, we may assume $(\wE _1 \cdot \wE _4)=0$ in what follows. 
Indeed, if $(\wE _1 \cdot \wE _4)=1$, we swap the roles of the pairs $(\wE _1,\wE _2)$ and $(\wE _4,\wE _3)$. 
Moreover, since $\wD _{\kc}$ is connected and has no cycle, we may assume that $\wE _2$ and $\wE _3$ are included in the same connected component of $\wD _{\kc} -(\wE _1+\wE _4)$. 
Then we can assume that $\nu$ starts with $\wE _1+\wE _4$, so that $\nu _{\ast}(\wE _2)$ and $\nu _{\ast}(\wE _3)$ are curves on $\cV$ with self-intersection number $\ge 0$. 
Hence, $(\nu _{\ast}(\wE _2) \cdot \nu _{\ast}(\wE _3))=1$ by Lemma \ref{Morrow}(3) and (4). 
Meanwhile, by Lemma \ref{lem(2-1)}, we obtain $(\wE _2 \cdot \wE _3)=1$. 
Moreover, by symmetry of the weighted dual graph of $\wD _{\kc}$, we can assume further $(\nu _{\ast}(\wE _2))^2 = (\nu _{\ast}(\wE _3))^2$. 
By Lemma \ref{Morrow}(3) and (4), we then obtain $\sharp \cD = 2$ and $(\nu _{\ast}(\wE _2))^2 = (\nu _{\ast}(\wE _3))^2 = 0$. 
Thus, we obtain this assertion. 
\end{proof}
%%%%%%%%%%%%%%%%%
We will prove Lemma \ref{lem(6-1)}, which plays an important role in \S \ref{6}. 
In order to show this lemma, we shall prepare the following lemma: 
%%%%%%%%%%%%%%%%%
\begin{lem}\label{lem(4-4)}
Let the notation be the same as above, and let $\wE _1$ and $\wE _2$ be two $(-1)$-curves in $\Supp (\wD _{\kc})$ meeting transversely at a point, say $x$. 
Assume that they are not defined over $\bk$ and any $(-1)$-curve lies on the $\Gal$-orbit of either $\wE _1$ or $\wE _2$. 
Then: 
\begin{enumerate}
\item Any closed point other than $x$ on $\wV _{\kc}$, which lies in the $\Gal$-orbit of $x$, is not included in $\wE_1 \cup \wE_2$.  
\item The $\Gal$-orbit of $x$ consists of at most two closed points on $\wV _{\kc}$. 
\item The weighted dual graph of $\wD _{\kc}$ is the twig $[1,1,1,1]$. 
\end{enumerate}
\end{lem}
%%%%%%%%%%%%%%%%%
\begin{proof}
By the assumption and Lemma \ref{Morrow}(1), we note $(\wE _1 \cdot \wD -\wE _1) \le 2$ or $(\wE _2 \cdot \wD -\wE _2) \le 2$. 

In (1), suppose that there exists a closed point $y$ other than $x$ on $\wE _1 \cup \wE _2$ lying on the $\Gal$-orbit of $x$. 
Without loss of generality, we may assume $y \in \wE _1 \backslash \{ x\}$. 
Then there exists a $(-1)$-curve $\wE _3$ other than $\wE _2$ in $\Supp (\wD _{\kc})$ meeting $\wE _1$ such that $\wE _2$ and $\wE _3$ lie in the same $\Gal$-orbit. 
Since $\wE _1$ is not defined over $\bk$, there exists a $(-1)$-curve $\wE _4$ in $\Supp (\wD _{\kc})$ other than $\wE _1$ such that $\wE _1$ and $\wE _4$ lie in the same $\Gal$-orbit. 
By Lemma \ref{lem(2-1)}, there exists a $(-1)$-curve $\wE _5$ in $\Supp (\wD _{\kc})$ meeting $\wE _4$ such that $\wE _2$ and $\wE _5$ lie in the same $\Gal$-orbit, $\wE _5 \not= \wE _2,\wE _3$ and $\wE _2+\wE _3+\wE _5$ is disjoint. 
Now, if $(\wE _1 \cdot \wD -\wE _1) \le 2$, by noting $(\wE _1 \cdot \wE _4)=0$, the direct image of $\wD _{\kc}$ by the contraction of $
\wE _1+\wE _4$ contains at least three irreducible components with self-intersection number $\ge 0$. It contradicts Lemma \ref{Morrow}(4). 
In what follows, we treat the case $(\wE _1 \cdot \wD -\wE _1) \ge 3$.
Then notice $(\wE _2 \cdot \wD -\wE _2) \le 2$. 
Since $\wE _1$ meets $\wE _2$ and $\wE _3$, there exists a $(-1)$-curve $\wE _6$ in $\Supp (D_{\kc})$ meeting $\wE _4$ such that $\wE _2$ and $\wE _6$ lie in the same $\Gal$-orbit, and $\wE _6 \not= \wE _2,\wE _3,\wE _5$. 
If $\wE _2+\wE _3+\wE _5+\wE _6$ is disjoint, the direct image of $\wD _{\kc}$ by the contraction of $\wE _2+\wE _3+\wE _5+\wE _6$ contains at least two irreducible components with self-intersection number $\ge 1$. 
Otherwise, by noticing $(\wE _6 \cdot \wE _2+\wE _3)=1$, the direct image of $\wD _{\kc}$ by the contraction of $\wE _2+\wE _3+\wE _5$ contains at least three irreducible components with self-intersection number $\ge 0$. 
Both cases contradict Lemma \ref{Morrow}. 

In (2), let $x_1,\dots ,x_r$ be all closed points lying on the $\Gal$-orbit of $x$. 
By (1) and the assumption, $\Supp (\wD _{\kc})$ contains exactly $2r$-times of $(-1)$-curves $\wE _1,\dots ,\wE _{2r}$. 
We may assume that $(\wE _1 \cdot \wD -\wE _1) \le 2$ and $x_i$ lies on $\wE _{2i-1}$ and $\wE _{2i}$ for $i=1,\dots ,r$. 
Suppose $r \ge 3$. 
Then we may assume that $\wE _3$ and $\wE _5$ lie on the $\Gal$-orbit. 
If $\wE _1+\wE _3+\wE _5$ is disjoint, the direct image of $\wD _{\kc}$ by the contraction of $\wE _1+\wE _3+\wE _5$ contains at least three irreducible components with self-intersection number $\ge 0$. It contradicts \ref{Morrow}(4). 
Otherwise, we may assume $(\wE _1 \cdot \wE _3)=1$. 
Then $(\wE _1 \cdot \wE _5)=0$ since $\wE _1$ meets only $\wE _2$ and $\wE _3$ on $\Supp (\wD _{\kc})$, so that the direct image of $\wD _{\kc}$ by the contraction of $\wE _1+\wE _5$ contains at least three irreducible components with self-intersection number $\ge 0$. It contradicts Lemma \ref{Morrow}. 

In (3), note that the closed point $x$ on $\wV _{\kc}$ is not defined over $\bk$. 
Indeed, otherwise, by the assumption, all $(-1)$-curves in $\Supp (\wD _{\kc})$ are only $\wE _1$ and $\wE _2$, moreover, they lie in the same $\Gal$-orbit. However, it is impossible by Lemma \ref{lem(4-1)}. 
Hence, $(V,D)$ satisfies the assumption of Lemma \ref{lem(4-2)} by (1) and (2). 
Thus, this assertion follows from Lemma \ref{lem(4-2)}. 
\end{proof}
%%%%%%%%%%%%%%%%%
\begin{lem}\label{lem(6-1)}
Let the notation be the same as above, and let $\wE _1,\dots ,\wE _r$ be all $(-1)$-curves in $\Supp (\wD _{\kc})$. 
Assume that they lie in the same $\Gal$-orbit and the union $\wE := \sum _{i=1}^r\wE _i$ is disjoint. 
Hence, we obtain the contraction $\nu ':\wV \to \wV '$ of $\wE$ defined over $\bk$ by Lemma \ref{Morrow}(1). 
Put $\wD ' := \nu '_{\ast}(\wD)$. 
Then one of the following proprieties holds: 
\begin{enumerate}
\item $(\wV '_{\kc},\wD '_{\kc})$ is an mnc of $\bA ^2_{\kc}$. 
\item The weighted dual graph of $\wD '_{\kc}$ is the twig either $[1,1,m]$ for some $m \ge 1$ or $[1,1,1,1]$. 
\item Letting $\wE _1',\dots ,\wE _{r'}'$ be all $(-1)$-curves on $\Supp (\wD _{\kc}')$, they lie in the same $\Gal$-orbit and the union $\sum _{i=1}^{r'}\wE _i'$ is disjoint.
\end{enumerate}
\end{lem}
%%%%%%%%%%%%%%%%%
\begin{proof}
By the assumption, each $x_i := \nu _{\kc}'(E_i)$ is contained in the same $\Gal$-orbit for $i=1,\dots ,r$. 
Moreover, $x_1$ lies in at most two $(-1)$-curves in $\Supp (\wD _{\kc}')$ since $(\wE _1 \cdot \wD - \wE _1) \le 2$. 

Assume that $x_1$ lies in no $(-1)$-curve in $\Supp (\wD _{\kc}')$. 
Then $\Supp (\wD _{\kc}')$ contains no $(-1)$-curve. 
Hence, $(\wV '_{\kc},\wD '_{\kc})$ is an mnc of $\bA ^2_{\kc}$. 

Assume that $x_1$ lies in exactly two $(-1)$-curves, say $\wE _1'$ and $\wE _2'$, in $\Supp (\wD _{\kc}')$. 
If $\wE _1'$ and $\wE _2'$ are not defined over $\bk$, the weighted dual graph of $\wD '_{\kc}$ is the twig $[1,1,1,1]$ by Lemma \ref{lem(4-4)}. 
In what follows, we thus consider the case that $\wE _1'$ is defined over $\bk$. 
Then $x_1,\dots ,x_r \in \wE _1'$ since these points lie in the same $\Gal$-orbit. 
Thus, $(\wE _1' \cdot \wD ' - \wE _1') \le 2$ by Lemma \ref{lem(4-3)}. In particular, $r \le 2$. 
If $r=1$, then $\Supp (\wD _{\kc}')$ contains only two $(-1)$-curves $\wE _1'$ and $\wE _2'$. 
Hence, the weighted dual graph of $\wD _{\kc}'$ is the twig $[1,1,m]$ for some $m \ge 2$ by Lemma \ref{lem(4-1)}. 
If $r=2$, letting $\nu '':\wV _{\kc}' \to \wV ''$ be the contraction of $\wE _1'$ over $\kc$, then $(\wV '',\wD '')$ is an mnc of $\bA ^2_{\kc}$. 
Moreover, $\Supp (\wD '')$ contains exactly two $0$-curves, so that the weighted dual graph of $\wD ''$ is the twig $[0,0]$ by Lemma \ref{Morrow}(4). 
Thus, the weighted dual graph of $\wD _{\kc}'$ is the twig $[1,1,1]$. 

Assume that $x_1$ lies in exactly one $(-1)$-curve in $\Supp (\wD _{\kc}')$. 
Let $\wE _1',\dots ,\wE _{r'}'$ be all $(-1)$-curves on $\Supp (\wD _{\kc}')$. 
Then each $x_i$ lies in exactly one $(-1)$-curve in $\Supp (\wD _{\kc}')$ for $i=1,\dots ,r$, so that $\wE _1',\dots ,\wE _{r'}'$ lie in the same $\Gal$-orbit. 
In what follows, we will show that the union $\sum _{i=1}^{r'}\wE _i'$ is disjoint.
Suppose that $\wE _1'$ and $\wE _2'$ transversely meet at a point, say $x$. 
By the assumption, we obtain $x \not= x_i$ for any $i=1,\dots ,r$. 
Meanwhile, we obtain $r'>2$. 
Indeed, otherwise the weighted dual graph of $\wD _{\kc}'$ is the twig $[1,1,m]$ for some $m \ge 2$ by Lemma \ref{lem(4-1)}, however, it contradicts that $\wE _1'$ and $\wE  _2'$ lie in the same $\Gal$-orbit. 
Noticing that the $\Gal$-orbit of $\wE _1'+\wE _2'$ has no cycle by Lemma \ref{lem(2-1)}, we thus have $r' \ge 4$ and the union $\sum _{i=1}^{r'}\wE _i'$ is not connected. 
Hence, we may assume $(\wE _3' \cdot \wE _4')=1$. 
Now, note $(\wE _1'+\wE _2' \cdot \wE _3'+\wE _4')=0$. 
On the other hand, since $\wD _{\kc}'$ is connected, we may assume that there exists a connected divisor $\wD _{1,3}'$ on $\wV _{\kc}'$ such that $\Supp (\wE _1' + \wE _3') \subseteq \Supp (\wD _{1,3}') \subseteq \Supp (\wD _{\kc}')$. 
Moreover, since $\wE _1'+\wE _2'$ and $\wE _3'+\wE _4'$ lie in the same $\Gal$-orbit, there exists a connected divisor $\wD _{2,4}'$ lying the $\Gal$-orbit of $\wD _{1,3}'$ such that $\Supp (\wE _2' + \wE _4') \subseteq \Supp (\wD _{2,4}') \subseteq \Supp (\wD _{\kc}')$. 
Then $\wD _{1,3}'+\wD _{2,4}'$ has a cycle. 
This contradicts Lemma \ref{lem(2-1)}. 
Thus, $\sum _{i=1}^{r'}\wE _i'$ is a disjoint union. 
\end{proof}
%%%%%%%%%%%%%%%%%%%%%%%%%%%%%%%%%%%%%%%%%%%%%%%%%%%%%%%%%%%%%%%%%%%%%%%%%%%%%%%%%%%%%%%%%%%%%%%%%%%%%%%%%%
\section{Proof of Theorem \ref{main}(1) and (2)}\label{5}
%%%%%%%%%%%%%%%%%
Let $(S,\Delta )$ be a compactification of $\bA ^2_{\bk}$ over $\bk$. 
Assume that $S$ is a numerical del Pezzo surface of rank one over $\bk$ such that $\Sing (S_{\kc}) \not= \emptyset$. 
Let $\{ C_i\}_{1 \le i \le n}$ be all irreducible components of the divisor $\Delta _{\kc}$ on $S_{\kc}$. 
By Lemma \ref{lem(2-2)}, $n=\rho _{\kc}(S_{\kc})$, and $C_1,\dots ,C_n$ lie in the same $\Gal$-orbit. 
Let $\sigma : \wS \to S$ be the minimal resolution defined over $\bk$, let $\wDelta$ be the divisor on $\wS$ defined by $\wDelta := \sigma ^{\ast}(\Delta )_{\red}$ and let $\wC _i$ be the proper transform of $C_i$ by $\sigma _{\kc}$ for $i=1,\dots ,n$. 
Let $\mu :\hS \to \wS$ be the composite of the shortest sequence of blow-ups such that $\hDelta _{\kc} = \mu _{\kc}^{\ast}(\wDelta _{\kc})_{\red}$ is an SNC-divisor (see also {\cite[Lemma 4.1]{KT09}}). 
Notice that $\mu$ is defined over $\bk$. 
Hence, the compactification $(\hS ,\hDelta )$ of the affine plane is defined over $\bk$. 
Now, let $\hDelta _{\mu}$ be the reduced exceptional divisor of $\mu _{\kc}$, let $\hDelta _{\sigma}$ be the proper transform of the reduced exceptional divisor of $\sigma _{\kc}$ by $\mu _{\kc}$, and let $\hC _i$ be the proper transform of $\wC _i$ by $\mu _{\kc}$ for $i=1,\dots ,n$. 
Namely, $\hDelta _{\kc} = \hDelta _{\mu} + \hDelta _{\sigma} + \sum _{i=1}^n\hC _i$. 
%%%%%%%%%%%%%%%%%
The purpose of this section is to prove Theorem \ref{main}(1) and (2). 
For these assertions, the case of $n=1$ is mostly based on the argument of {\cite{Ko01,KT09}}, on the other hand, in order to deal with the case of $n \ge 2$, we need to observe the behavior of the Galois group $\Gal$ acting naturally on $\hS _{\kc}$. 
Hence, we shall treat such observations in \S \S \ref{5-1}, and we shall show Theorem \ref{main}(1) and (2) in \S \S \ref{5-2} and \S \S \ref{5-3}, respectively. 
%%%%%%%%%%%%%%%%%%%%%%%%%%%%%%%%%%%%%%%%%%%%%%%%%%%%%%%%
\subsection{Some observations in case of $n \ge 2$}\label{5-1}
%%%%%%%%%%%%%%%%%
With the same notation as above, assume further that $n \ge 2$. 
%%%%%%%%%%%%%%%%%

By using Lemma \ref{Morrow} and construction of $\mu$, the following lemma is obvious: 
%%%%%%%%%%%%%%%%%
\begin{lem}\label{easy lem}
With the notation and assumptions as above, the following two assertions hold: 
\begin{enumerate}
\item Any irreducible component of $\hDelta _{\sigma}$ (resp. $\hDelta _{\mu}$) has self-intersection number $\le -2$ (resp. $\le -1$). 
\item Assume that $\mu \not= id$. 
Then, $\hDelta _{\mu}$ contains at least one $(-1)$-curve. 
Moreover, any $(-1)$-curve $\hE$ on $\hDelta _{\mu}$ does not meet any $(-1)$-curve on $\hDelta _{\mu}$, and satisfies both $(\hE \cdot \hDelta - \hE ) \ge 3$ and $\sum _{i=1}^n(\hE \cdot \hC _i)>0$. 
\end{enumerate}
\end{lem}
%%%%%%%%%%%%%%%%%
By observing the behavior of the $\Gal$-action, we obtain the following lemma: 
%%%%%%%%%%%%%%%%%
\begin{lem}\label{not mnc}
With the notation and assumptions as above, $(\hS _{\kc}, \hDelta _{\kc})$ is not an mnc of $\bA ^2_{\kc}$. 
In particular, $\hC _i$ is a $(-1)$-curve with $(\hC _i \cdot \hDelta - \hC _i) \le 2$ for $i=1,\dots ,n$. 
\end{lem}
%%%%%%%%%%%%%%%%%
\begin{proof}
Suppose that $(\hS _{\kc}, \hDelta _{\kc})$ is an mnc of $\bA ^2_{\kc}$. 
By $n \ge 2$ and $\Sing (S_{\kc}) \not= \emptyset$, we have $\sharp \hDelta _{\kc} \ge 3$. 
By Lemma \ref{Morrow}(4), there exist two irreducible components $\Gamma _0$ and $\Gamma _+$ on $\hDelta _{\kc}$ such that $(\Gamma _0)^2 =0$ and $(\Gamma _+)^2>0$. 
By Lemma \ref{easy lem}(1), $\Gamma _0$ and $\Gamma _+$ are contained in $\sum _{i=1}^n\hC _i$. 
However, we see $(\hC _1)^2=\dots = (\hC _n)^2$ since $\hC _1,\dots ,\hC _n$ lie in the same $\Gal$-orbit. 
It is a contradiction. 
\end{proof}
%%%%%%%%%%%%%%%%%
Since $\Delta _{\kc}$ has no cycle by Lemma \ref{lem(2-1)}, $C_1,\dots ,C_n$ meet only one point, say $p_0$. 
Then we obtain the following lemma: 
%%%%%%%%%%%%%%%%%
\begin{lem}\label{lem(5-1)}
Let the notation and the assumptions be the same as above. 
If $n=2$, then $p_0$ is a singular point on $S_{\kc}$. 
\end{lem}
%%%%%%%%%%%%%%%%%
\begin{proof}
Suppose that $p_0$ is a non-singular point on $S_{\kc}$. 
Then we obtain $\sharp \hDelta _{\kc} \ge 4$ since $S_{\kc}$ has at least two singular points by Lemma \ref{lem(2-1)}. 
Moreover, we see $(\hC _1 \cdot \hC _2)=1$. 
On the other hand, $\hC _1 $ and $\hC _2$ are $(-1)$-curves by Lemma \ref{not mnc}. 
It is a contradiction to Lemma \ref{lem(4-1)}. 
\end{proof}
%%%%%%%%%%%%%%%%%%%%%%%%%%%%%%%%%%%%%%%%%%%%%%%%%%%%%%%%
\subsection{Properties of boundary divisors}\label{5-2}
%%%%%%%%%%%%%%%%%
In this subsection, we prove Theorem \ref{main}(1). 
Since the case of $n=1$ follows from {\cite{Ko01,KT09}}, we shall only treat case of $n \ge 2$. 
Let the notation be the same as above. 
In what follows, suppose on the contrary that $\mu \not= id$. 
%%%%%%%%%%%%%%%%%
\begin{claim}\label{claim}
The union $\sum _{i=1}^n\hC _i$ is disjoint. 
\end{claim}
%%%%%%%%%%%%%%%%%
\begin{proof}
Suppose that $\hC _i$ and $\hC _j$ meet at a point, say $\hat{q}$, for some $i$ and $j$ with $i \not= j$. 
Since $\hDelta _{\kc}$ is an SNC-divisor, any irreducible component of $\hDelta _{\kc}$ passing through $\hat{q}$ is only $\hC _i$ or $\hC _j$. 
Hence, $q := (\sigma _{\kc}\circ \mu _{\kc})(\hat{q})$ is a smooth point on $S_{\kc}$, moreover, $C_i$ and $C_j$  pass through $q$. 
Let $p_0$ be the intersection point of $C_1,\dots ,C_n$ on $S_{\kc}$. 
By noting Lemma \ref{lem(5-1)}, we see $p_0 \not= q$. 
However, $C_i+C_j$ is then a cycle on $\Delta _{\kc}$. 
It is a contradiction to Lemma \ref{lem(2-1)}. 
\end{proof}
%%%%%%%%%%%%%%%%%
\begin{proof}[Proof of Theorem \ref{main}(1)]
By Lemma \ref{not mnc} and Claim \ref{claim}, $\sum _{i=1}^n\hC _i$ can be contracted. 
Let $\nu :\hS _{\kc} \to \cS$ be a sequence of contractions of $(-1)$-curves and subsequently (smoothly) contractible curves in $\Supp (\hDelta )$ over $\kc$, starting with the contraction of $\sum _{i=1}^n\hC _i$, such that $(\cS,\cDelta)$ is an mnc of $\bA ^2_{\kc}$, where $\cDelta := \nu _{\ast}(\hDelta _{\kc})$. 
For any $(-1)$-curve $\hE$ on $\hDelta _{\mu}$, we have $(\hE \cdot \hDelta -\hE) \ge 3$ by Lemma \ref{easy lem}, in particular, $\nu _{\ast}(\hE )\not= 0$ and $(\nu _{\ast}(\hE ))^2 \ge 0$. 
Hence, there exist at most two $(-1)$-curves on $\hDelta _{\mu}$ by Lemma \ref{Morrow}. 
By Lemma \ref{lem(4-3)}, we further see that there exist exactly two $(-1)$-curves $\hE _1$ and $\hE _2$ on $\hDelta _{\mu}$. 
Here, $\hE _1$ and $\hE _2$ lie in the same $\Gal$-orbit because two points $\mu (\hE _1)$ and $\mu (\hE _2)$ lie on $\bigcup _{i=1}^n\wC _i$. 
Noting Lemma \ref{easy lem}(2), we may assume that $\hE _i$ meets $\hC _i$ for $i=1,2$. 
Let $D$ be the connected component of the reduced exceptional divisor of $\nu$ containing $\hC _1$. 
Letting $\cE _i := \nu _{\ast}(\hE _i)$ for $i=1,2$, we obtain $(\hE _1 \cdot \hE _2)=0$ and $(\cE _1 \cdot \cE _2)=1$ by Lemmas \ref{easy lem}(2) and \ref{Morrow}, so that $(\hE _2 \cdot D)>0$. 
Hence, $\hC _2$ is included in $\Supp (D)$. 
Indeed otherwise, letting $\Gal \cdot D$ be the $\Gal$-orbit of $D$, then $\hE _1+\hE_2+\Gal \cdot D$ has a cycle. 
This contradicts Lemma \ref{lem(2-1)}. 
Moreover, we know $n=2$ by the similar argument. 
Thus, $(\cE _1 \cdot \cDelta - \cE _1) = (\hE _1 \cdot \hDelta - \hE _1) \ge 3$, which contradicts Lemma \ref{Morrow}(1). 
\end{proof}
%%%%%%%%%%%%%%%%%
\begin{eg}
Note that the assertion of Theorem \ref{main}(1) is not always true unless $\rho _{\bk}(S)=1$. 
We shall construct an example of the compactification $(S,\Delta )$ of $\bA ^2_{\bk}$ into an lc del Pezzo surface $S$ over $\bk$ such that $\rho _{\bk}(S)>1$ and $\wDelta _{\kc}$ is not an SNC-divisor. 
Let $C$ be a cubic curve with a cusp $o$ on $\bP ^2_{\bk}$ and let $L$ be the Zariski tangent line to $C$ at $o$, i.e., $C_{\kc} \cap L_{\kc} = \{ o\}$. 
By construction, $L \simeq \bP ^1_{\bk}$. 
Let $x_1$ be a $\bk$-rational point on $C_{\kc} \backslash \{ o\}$ and let $x_2,x_3,x_4$ be three points, whose union $x_1+x_2+x_3$ is defined over $\bk$, on $L_{\kc} \backslash \{ o\}$. 
Letting $\nu :\wS \to \bP ^2_{\bk}$ be a blow-up at four points $x_1,\dots ,x_4$ defined over $\bk$, then $\wS$ is a weak del Pezzo surface of degree $5$ such that $\wS _{\kc}$ contains exactly one $(-2)$-curve $\nu ^{-1}_{\ast}(L_{\kc})$. 
Let $\sigma :\wS \to S$ be a contraction of $\nu ^{-1}_{\ast}(L_{\kc})$ over $\bk$, so that $S$ is a Du Val del Pezzo surface with $\rho _{\bk}(S)>1$ over $\bk$. 
Now, $\nu$ can be factorized $\nu ':\wS \to \wS '$ and $\nu '':\wS ' \to \bP ^2_{\bk}$ defined over $\bk$ such that $\nu ''$ is a blow-up at a point $x_1$. 
Let $\wDelta '$ be the proper transform of $C+L$ by $\nu ''$, and let $\wDelta$ be the reduced effective divisor on $\wS$ defined by $\wDelta := {\nu '}^{\ast}(\wDelta ')_{\red}$. 
Since $(\wS ',\wDelta ')$ is a compactification of $\bA ^2_{\bk}$ ({\cite{B73}}), so are $(\wS ,\wDelta )$ and $(S,\Delta )$, where $\Delta := \sigma _{\ast}(\wDelta )$. 
By construction, $\wDelta _{\kc}$ is not an SNC-divisor. 
\end{eg}
%%%%%%%%%%%%%%%%%%%%%%%%%%%%%%%%%%%%%%%%%%%%%%%%%%%%%%%%
\subsection{Properties of singularities}\label{5-3}
%%%%%%%%%%%%%%%%%
In this subsection, we prove Theorem \ref{main}(2) by using results in \S \S \ref{5-1} and \S \S \ref{5-2}. 
Let the notation be the same as above. 
%%%%%%%%%%%%%%%%%
\begin{proof}[Proof of Theorem \ref{main}(2)(i)]
We shall consider two cases whether $n=1$ or not separately. 

In the case of $n=1$, then $\Delta$ is geometrically irreducible on $S$. 
Namely, $\Delta _{\kc} = C_1$. 
Hence, we see $\sharp \Sing (S_{\kc}) \le 2$ by Lemma \ref{lem(2-3)}. 
If $\sharp \Sing (S_{\kc}) =1$, then the assertion is clearly true. 
If $\sharp \Sing (S_{\kc}) =2$, then two weighted dual graphs given by the minimal resolution at these singular points on $S_{\kc}$ are different (see {\cite{Ko01,KT09}}), so that they are $\bk$-rational. 

In the case of $n\ge 2$, let $p_0$ be the intersection point of $C_1,\dots ,C_n$ on $S_{\kc}$, so that $p_0$ is $\bk$-rational. 
If $n=2$, then $p_0$ is a singular point on $S_{\kc}$ by Lemma \ref{lem(5-1)}. 
Moreover, $p_0$ is also a singular point on $S_{\kc}$ even if $n \ge 3$. 
Otherwise, the divisor $\wDelta _{\kc}$ is not normal crossing at the point $\widetilde{p}_0 := \sigma ^{-1}(p_0)$, which is a contradiction to Theorem \ref{main}(1). 
\end{proof}
%%%%%%%%%%%%%%%%%
\begin{proof}[Proof of Theorem \ref{main}(2)(iii)]
If $n = 1$, then it follows from Lemma \ref{lem(2-3)}. 
Indeed, $n=\rho _{\kc}(S_{\kc})$ by Lemma \ref{lem(2-2)}. 
Hence, we assume $n \ge 2$ in what follows. 
Let $p_0$ be the intersection point of $C_1,\dots ,C_n$ on $S_{\kc}$. 
Then we notice that $p_0$ is $\bk$-rational and singular on $S_{\kc}$ (see Proof of Theorem \ref{main}(2)(i)). 
Suppose $\sharp \Sing (S_{\kc}) >1$. 
Noting Lemma \ref{lem(2-1)}, there exist $n$-times of singular points $p_1,\dots ,p_n$ on $S_{\kc}$, which lie in the same $\Gal$-orbit, such that $p_i \in C_i \backslash \{ p_0\}$ for $i=1,\dots ,n$. 
For each $i=1,\dots ,n$, there exist two irreducible components meeting $\wC _i$ on $\wDelta _{\kc} - \left( \sum _{i=1}^n\wC_i \right)$ such that the images of these via $\sigma _{\kc}$ are two points $p_0$ and $p_i$. 
On the other hand, $(\wC _i \cdot \wDelta - \wC _i ) \le 2$ by Lemma \ref{not mnc}, so that $(\wC _i \cdot \wDelta - \wC _i ) = 2$, which implies $\Sing (S_{\kc}) \cap C_i = \{ p_0,\ p_i\}$. 
Therefore, $\sharp \Sing (S_{\kc}) = n+1 = \rho _{\kc}(S_{\kc})+1$. 
\end{proof}
%%%%%%%%%%%%%%%%%
\begin{proof}[Proof of Theorem \ref{main}(2)(ii)]
By Theorem \ref{main}(2)(i), we can take a singular point on $S_{\kc}$, which is $\bk$-rational, say $p_0$. 
If $\sharp \Sing (S_{\kc}) \le 2$, we see $\sharp \Sing (S_{\kc}) = \sharp \Sing (S)$. 
In what follows, we may assume $\sharp \Sing (S_{\kc}) \ge 3$. 
Then all singular points except for $p_0$ on $S_{\kc}$ lie in the same $\Gal$-orbit (see Proof of Theorem \ref{main}(2)(iii)).  
Thus, $\Sing (S) = \{ p_0\}$. 
This completes the proof. 
\end{proof}
%%%%%%%%%%%%%%%%%%%%%%%%%%%%%%%%%%%%%%%%%%%%%%%%%%%%%%%%%%%%%%%%%%%%%%%%%%%%%%%%%%%%%%%%%%%%%%%%%%%%%%%%%%
\section{Proof of Theorem \ref{main}(3)}\label{6}
Let $(S,\Delta )$ be a compactification of $\bA ^2_{\bk}$ over $\bk$. 
Assume that $S$ is an lc del Pezzo surface of rank one over $\bk$ such that $\Sing (S_{\kc}) \not= \emptyset$. 
Let $\sigma : \wS \to S$ be the minimal resolution over $\bk$ and let $\wDelta$ be the divisor on $\wS$ defined by $\wDelta := \sigma ^{\ast}(\Delta)_{\red}$. 
By Theorem \ref{main}(1), $\wDelta _{\kc}$ is an SNC-divisor. 
In this section, we classify the weighted dual graph of $\wDelta _{\kc}$. 
In fact, if $\sharp \Delta _{\kc} =1$, it is classified by {\cite{Ko01,KT09}}. 
Hence, assume $\sharp \Delta _{\kc} >1$ in what follows. 
The strategy in this case is mainly based on the argument in {\cite{Ko01,KT09}}, however, we need to use some results in \S \ref{4}. 

Let $\{ C_i\}_{1 \le i \le n}$ be all irreducible components of the divisor $\Delta _{\kc}$ on $S_{\kc}$. 
We note that $n \ge 2$ and $C_1,\dots ,C_n$ lie in the same $\Gal$-orbit. 
Let $\wC _i$ be the proper transform of $C_i$ by $\sigma _{\kc}$. 
Now, we shall construct a birational morphism $\nu :\wS \to \cS$ defined over $\bk$ such that $\wS \backslash \Supp (\wDelta ) \simeq \cS \backslash \Supp (\cDelta )$ and $(\cS _{\kc}, \cDelta _{\kc})$ is an mnc of $\bA ^2_{\kc}$, where $\cDelta  := \nu _{\ast}(\wDelta )$. 
By Lemma \ref{not mnc}, any $\wC _i$ is a $(-1)$-curve on $\wS _{\kc}$, moreover, $E_1' := \sum _{i=1}^n\wC _i$ is a disjoint union. 
Let $\cont _{E_1'}: \wS _1 := \wS _{\kc} \to \wS _2$ be the contraction of $E_1'$ defined over $\bk$. Then $\wS _1 \backslash \Supp (\wDelta _1) \simeq \wS _2 \backslash \Supp (\wDelta _2)$, where $\wDelta _1 := \wDelta _{\kc}$ and $\wDelta _2 := \cont _{E_1',\ast}(\wDelta _1)$. 
If $(\wS _{2,\kc}, \wDelta _{2,\kc})$ satisfies the property (1) in Lemma \ref{lem(6-1)}, we may put $\nu := \cont _{E_1'}$. 
If $(\wS _{2,\kc}, \wDelta _{2,\kc})$ satisfies the property (2) in Lemma \ref{lem(6-1)}, letting $E_2'$ be the union of $(-1)$-curves, which are all terminal component of $\wDelta _{2,\kc}$, we obtain the contraction $\cont _{E_2'}: \wS _2 \to \wS _3$ of $E_2'$ defined over $\bk$, so that $\wS _2 \backslash \Supp (\wDelta _2) \simeq \wS _3 \backslash \Supp (\wDelta _3)$ and $(\wS _{3,\kc}, \wDelta _{3,\kc})$ is an mnc of $\bA ^2_{\kc}$, where $\wDelta _3 := \cont _{E_2',\ast}(\wDelta _2)$. 
Hence, we may put $\nu := \cont _{E_2'} \circ \cont _{E_1'}$. 
Otherwise, by using Lemma \ref{lem(6-1)} repeatably, we can construct a sequence of contractions $\cont _{E_{\ell '}'} \circ \dots \circ \cont _{E_2'}: \wS _2 \to \dots \to \wS _{\ell '+1}$ defined over $\bk$ such that $\wS _2 \backslash \Supp (\wDelta _2) \simeq \cdots \simeq \wS _{\ell '+1} \backslash \Supp (\wDelta _{\ell '+1})$ and $(\wS _{\ell '+1,\kc}, \wDelta _{\ell '+1,\kc})$ satisfies either property (1) or (2) in Lemma \ref{lem(6-1)}, where each $E_i'$ is the disjoint union of all $(-1)$-curves in $\Supp (\wDelta _{i,\kc})$ defined over $\bk$, and $\wDelta _{i+1} := \cont _{E_i',\ast}(\wDelta _i)$ for $i=2,\dots ,\ell '$. 
Hence, we obtain the birational morphism $\nu :\wS \to \cS$ defined over $\bk$ such that $\wS \backslash \Supp (\wDelta ) \simeq \cS \backslash \Supp (\cDelta )$ and $(\cS _{\kc}, \cDelta _{\kc})$ is an mnc of $\bA ^2_{\kc}$, where $\cDelta  := \nu _{\ast}(\wDelta )$. 
By construction of $\nu$, we obtain the following lemma: 
%%%%%%%%%%%%%%%%%
\begin{lem}[cf.\ {\cite[Lemma 4.5]{KT09}}]\label{lem(6-2)}
With the notation and assumptions as above, then we obtain $\sharp \cDelta _{\kc} \le 2$. 
More precisely, $(\cS _{\kc},\cDelta _{\kc})$ is either $(\bP ^2_{\kc},L)$ or $(\bF _m,M_m+F)$ for some non-negative integer $m \not= 1$, where $L$ is a line on $\bP ^2_{\kc}$ and $M_m$ (resp. $F$) is the minimal section (resp. a fiber) of the structure morphism $\bF _m \to \bP ^1_{\kc}$. 
\end{lem}
%%%%%%%%%%%%%%%%%
\begin{proof}
Suppose that $\sharp \cDelta _{\kc} \ge 3$. 
Since $(\cS _{\kc},\cDelta _{\kc})$ is an mnc of $\bA ^2_{\kc}$, we see that $\cDelta _{\kc}$ contains two components $\Gamma _0$ and $\Gamma _+$ such that $(\Gamma _0)^2=0$, $(\Gamma _+)^2>0$ and $(\Gamma _0 \cdot \Gamma _+)=1$ by Lemma \ref{Morrow}(4). 
Moreover, $\Gamma _0$ and $\Gamma _+$ are defined over $\bk$. 
Noticing that $\nu$ is defined over $\bk$, so is $q := \nu \left( \sum _{i=1}^n\wC_i \right)$. 
Since any irreducible component of $\wDelta _{\kc} -\sum _{i=1}^n\wC _i$ has self-intersection number $\le -2$, we see that $q$ is the intersection point of $\Gamma _0$ and $\Gamma _+$. 
However, we then have $(\nu ^{-1}_{\ast}(\Gamma _0))^2 \ge -1$ or $(\nu ^{-1}_{\ast}(\Gamma _+))^2 \ge -1$. 
This is a contradiction. 
\end{proof}
%%%%%%%%%%%%%%%%%
Since $\nu$ is defined over $\bk$, notice that $\nu$ can be written by the composition of contractions $\nu = \cont _{E_{\ell}} \circ \dots \circ \cont _{E_1}$, where each $E_i$ is a $\Gal$-orbit of a $(-1)$-curve, which is a disjoint union. 
Hence, let us put $\nu _i := \cont _{E_{\ell}} \circ \dots \circ \cont _{E_i}$ for $i=1,\dots ,\ell$. 
By Lemma \ref{lem(6-2)}, we see that $(E_{i,j} \cdot \nu _i^{\ast}(\cDelta )_{\red} - E_{i,j})$ is equal to $1$ or $2$ for any irreducible component $E_{i,j}$ of $E_i$. 

By Theorem \ref{main}(2) and Lemma \ref{lem(2-2)}, $S_{\kc}$ has exactly one singular point $p_0$, which is $\bk$-rational, and $\sharp \Sing (S_{\kc}) =n+1$ or $1$. 
Let $p_1,\dots ,p_n$ be singular points other than $p_0$ on $S_{\kc}$ such that $p_i \in C_i$ for $i=1,\dots ,n$ (if it exists) and let $\wDelta ^{(i)}$ be the reduced exceptional divisor of the minimal resolution at $p_i$ on $S_{\kc}$ for $i=0,\dots ,n$, where $\wDelta ^{(i)}:=0$ if $p_i$ does not exist. 
Namely, $\wDelta _{\kc} = \sum _{i=0}^n \wDelta ^{(i)} + \sum _{i=1}^n\wC _i$. 
By the above argument, we have: 
%%%%%%%%%%%%%%%%%
\begin{lem}[cf.\ {\cite[Lemma 4.6]{KT09}}]\label{lem(6-3)}
Let the notation and the assumptions be the same as above. 
For any $i=1,\dots ,n$, the following three assertions hold: 
\begin{enumerate}
\item If $\sharp \Sing (S_{\kc}) =n+1$, then the dual graph of $\wDelta ^{(i)}$ is a linear chain. 
\item If $\sharp \Sing (S_{\kc}) =n+1$, then $\wC _i$ meets a terminal component of $\wDelta ^{(i)}$. 
\item Any irreducible component $\wGamma _0$ of $\wDelta ^{(0)}$ with $(\wGamma _0 \cdot \wDelta ^{(0)} - \wGamma _0) \ge 3$ does not meet $\wC _i$. 
\end{enumerate}
\end{lem}
%%%%%%%%%%%%%%%%%
Now, the singular point $p_0$ has the following three possibilities: 
\begin{enumerate}
\renewcommand{\labelenumi}{(\Roman{enumi})}
\item $p_0$ is a cyclic quotient singular point; 
\item $p_0$ is a non-cyclic quotient singular point; 
\item $p_0$ is a log canonical but not a quotient singular point. 
\end{enumerate}
In order to determine the weighted dual graph of $\wDelta _{\kc}$, we will consider the above three cases (I)--(III) separately according to the following subsections \S \S \ref{6-1}--\ref{6-3}. 
%%%%%%%%%%%%%%%%%%%%%%%%%%%%%%%%%%%%%%%%%%%%%%%%%%%%%%%%
\subsection{Case of cyclic quotient singularity}\label{6-1}
%%%%%%%%%%%%%%%%%
Assume that $p_0$ is a cyclic quotient singular point. 
Then we consider the following three subcases separately: 
%%%%%%%%%%%%%%%%%
\begin{enumerate}
\renewcommand{\labelenumi}{(I-\arabic{enumi})}
\item $\wDelta _{\kc}$ is not a linear chain and any irreducible component of $\wDelta ^{(0)}$ is defined over $\bk$; 
\item $\wDelta _{\kc}$ is not a linear chain and there exists an irreducible component of $\wDelta ^{(0)}$, which is not defined over $\bk$; 
\item $\wDelta _{\kc}$ is a linear chain. 
\end{enumerate}
%%%%%%%%%%%%%%%%%
\subsubsection{Subcase (I-1)}\label{6-1-1}
Then there exists exactly one irreducible component $\wGamma$ of $\wDelta ^{(0)}$ such that $\sum _{i=1}^n(\wGamma \cdot \wC _i) =n$. 
In particular, $(\widetilde{\Gamma} \cdot \wDelta - \widetilde{\Gamma}) \ge 3$. 
Thus, $\nu$ must first repeat the contraction until all irreducible components in $\Supp \left( \sum _{i=1}^n\wDelta ^{(i)}\right)$ for $i=1,\dots ,n$ are contracted, so that $\wDelta ^{(i)}$ is a linear chain consisting entirely of $(-2)$-curves for $i=1,\dots ,n$ (see also Lemma \ref{lem(6-3)}). 
By the above argument combined with Lemma \ref{Morrow} and Proposition \ref{prop(3)}(1), the weighted dual graph of $\wDelta _{\kc}$ is given as $(i)$ $(i=1,2,3)$ in Appendix \ref{9-1}, where $n \ge 3$ in the case of (1). 
%%%%%%%%%%%%%%%%%
\subsubsection{Subcase (I-2)}\label{6-1-2}
Then there exist exactly two irreducible components $\wGamma _1$ and $\wGamma _2$ of $\wDelta ^{(0)}$ such that $(\wGamma _i \cdot \wDelta - \wGamma _i) \ge 3$ for $i=1,2$ by noting Lemma \ref{lem(3-4)}(1), in particular, $\wGamma _1$ and $\wGamma _2$ lie in the same $\Gal$-orbit. 
By a similar argument to \ref{6-1-1}, $\wDelta ^{(i)}$ is a linear chain consisting entirely of $(-2)$-curves for $i=1,\dots ,n$. 
Hence, by Lemma \ref{Morrow} and Proposition \ref{prop(3)}(2) and (3), the weighted dual graph of $\wDelta _{\kc}$ is given as $(i)$ $(i=4,5,6,7)$ in Appendix \ref{9-1}, where $n' := \frac{n}{2} \ge 2$ in the case of (4). 
%%%%%%%%%%%%%%%%%
\subsubsection{Subcase (I-3)}\label{6-1-3}
Then we immediately see that $n=2$ and both $\wC _1$ and $\wC _2$ meet terminal components in $\wDelta ^{(0)}$, respectively. 
Moreover, these terminal components in $\wDelta ^{(0)}$ are the same component if and only if $\sharp \wDelta ^{(0)}=1$. 
Thus, by Proposition \ref{prop(3)}(2) and (3), the weighted dual graph of $\wDelta _{\kc}$ is given as $(i)$ $(i=8,9,10)$ in Appendix \ref{9-1}. 
%%%%%%%%%%%%%%%%%%%%%%%%%%%%%%%%%%%%%%%%%%%%%%%%%%%%%%%%
\subsection{Case of non-cyclic quotient singularity}\label{6-2}
%%%%%%%%%%%%%%%%%
Assume that $p_0$ is a non-cyclic quotient singular point. 
In this case, Lemma \ref{lem(6-3)} and the following lemma play a useful role: 
%%%%%%%%%%%%%%%%%
\begin{lem}\label{lem(6-4)}
Let the notation and the assumptions be the same as above. 
For any irreducible component $\wGamma$ of $\wDelta ^{(0)}$ satisfying $n':= (\wGamma \cdot \wC _1 + \dots + \wC _n) >0$, then we have $(\wGamma )^2 < -n'$. 
\end{lem}
%%%%%%%%%%%%%%%%%
\begin{proof}
Suppose that $(\wGamma )^2 \ge -n'$. 
By the construction of $\cont _{E_1}$, we have $(\cont _{E_1,\ast}(\wGamma ))^2 \ge -n' + n'=0$. 
This means that any irreducible component of $\wDelta ^{(0)}$ is not contracted by $\nu$. 
Thus, $\cDelta$ is not a linear chain, which is a contradiction to Lemma \ref{Morrow}(1). 
\end{proof}
%%%%%%%%%%%%%%%%%
Note that the classification of quotient singularities of dimension two is well-known. 
In particular, a list of weighted dual graphs of the minimal resolutions at all quotient singularities on algebraic surfaces is summarized in, e.g., {\cite[p.\ 57]{A92}} or {\cite[pp.\ 55--56]{Mi01}}. 
%%%%%%%%%%%%%%%%%

By the assumption, $\wDelta ^{(0)}$ is not a linear chain, in particular, there is exactly one irreducible component $\wGamma _0$ of $\wDelta ^{(0)}$ such that $(\wGamma _0 \cdot \wDelta ^{(0)} - \wGamma _0)=3$. 
On the other hand, since $\cDelta$ is a linear chain by Lemma \ref{Morrow}(1), $\nu$ is factorized $\nu '' \circ \nu '$ such that $\wGamma _0' \not= 0$ and $(\wGamma _0' \cdot \wDelta '- \wGamma _0') \le 2$, where $\wGamma _0' := \nu _{\ast}'(\wGamma _0)$ and $\wDelta ' := \nu _{\ast}'(\wDelta _{\kc})$. 
We can assume that $\nu '$ is defined over $\bk$ since $\nu$ is a sequence of contractions of the $\Gal$-orbit consisting of $(-1)$-curves and subsequently $\Gal$-orbits consisting of (smoothly) contractible curves in $\Supp (\wDelta _{\kc})$. 
In other words, any irreducible component $\wDelta '$ except for $\wGamma _0'$ has self-intersection number $\le -2$. 
Hence, $\nu '$ is uniquely determined. 
%%%%%%%%%%%%%%%%%
Now, we consider the following three subcases (II-1)--(II-3) separately: 
%%%%%%%%%%%%%%%%%
\begin{enumerate}
\renewcommand{\labelenumi}{(II-\arabic{enumi})}
\item $(\wGamma _0' \cdot \wDelta '- \wGamma _0')=2$ holds; 
\item $(\wGamma _0' \cdot \wDelta '- \wGamma _0')=1$ holds; 
\item $(\wGamma _0' \cdot \wDelta '- \wGamma _0')=0$ holds. 
\end{enumerate}
%%%%%%%%%%%%%%%%%
\subsubsection{Subcase (II-1)}\label{6-2-1}
By virtue of $\sharp \wDelta ' \ge 3$ combined with Lemma \ref{Morrow}(4), $\wGamma _0'$ is a $(-1)$-curve. 
Thus, by Proposition \ref{prop(3)}(1) and Lemmas \ref{lem(6-3)} and \ref{lem(6-4)} combined with the classification of singularities, the weighted dual graph of $\wDelta _{\kc}$ is given as $(i)$ $(i=11,\dots ,17)$ in Appendix \ref{9-2}. 
%%%%%%%%%%%%%%%%%
\subsubsection{Subcase (II-2)}\label{6-2-2}
Then the weighted dual graph of $\wDelta ^{(0)}$ is one of the following three weighted dual graphs, where $m_i$ and $m$ are integers such that $m_i \ge 2$ and $m \ge 2$: 
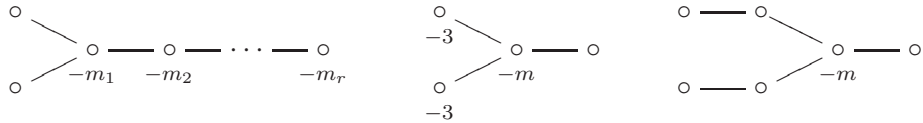
\begin{align*}
\xygraph{
\circ -[]!{+(.8,-.4)} \circ ([]!{+(0,-.3)} {^{-m_1}})
(-[]!{+(-.8,-.4)} \circ ,- []!{+(.8,0)} \circ ([]!{+(0,-.3)} {^{-m_2}}) - []!{+(.8,0)} \cdots - []!{+(.8,0)} \circ ([]!{+(0,-.3)} {^{-m_r}})
} \qquad
\xygraph{
\circ ([]!{+(0,-.3)} {^{-3}}) -[]!{+(.8,-.4)} \circ ([]!{+(0,-.3)} {^{-m}})
(-[]!{+(-.8,-.4)} \circ ([]!{+(0,-.3)} {^{-3}}),- []!{+(.8,0)} \circ )
} \qquad 
\xygraph{
\circ - []!{+(.8,0)} \circ ([]!{+(0,0)} {} -[]!{+(.8,-.4)} \circ ([]!{+(0,-.3)} {^{-m}})
(-[]!{+(-.8,-.4)} \circ - []!{+(-.8,0)} \circ ,- []!{+(.8,0)} \circ )
}
\end{align*}
Note that $\wGamma _0'$ is not a $(-1)$-curve by Lemma \ref{lem(3-1)}(2). 
Hence, we obtain $\nu ' = \nu$. 
Moreover, we see that $\sharp \wDelta ' =2$ and $\wGamma _0'$ is a $0$-curve by Lemma \ref{Morrow}(3) and (4). 
Thus,  by Lemmas \ref{lem(6-3)} and \ref{lem(6-4)}, the weighted dual graph of $\wDelta _{\kc}$ is given as $(i)$ $(i=18,\dots ,25)$ in Appendix \ref{9-2}. 
%%%%%%%%%%%%%%%%%
\subsubsection{Subcase (II-3)}\label{6-2-3}
Then the weighted dual graph of $\wDelta ^{(0)}$ is then as follows, where $m$ is an integer with $m\ge 2$: 
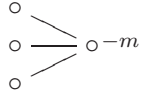
\begin{align*}
\xygraph{
\circ ([]!{+(.3,0)} {^{-m}}) (- []!{+(-.8,0)} \circ ,(- []!{+(-.8,.4)} \circ , - []!{+(-.8,-.4)} \circ)) }
\end{align*}
By the assumption, we see $\nu = \nu '$ and $\sharp \wDelta ' =1$. 
In particular, $\wGamma _0'$ is a $1$-curve by Lemma \ref{Morrow}(2). 
Thus, by Lemmas \ref{lem(6-3)} and \ref{lem(6-4)}, the weighted dual graph of $\wDelta _{\kc}$ is given as $(26)$ or $(27)$ in Appendix \ref{9-2}. 
%%%%%%%%%%%%%%%%%%%%%%%%%%%%%%%%%%%%%%%%%%%%%%%%%%%%%%%%
\subsection{Case of log canonical but not quotient singularity}\label{6-3}
%%%%%%%%%%%%%%%%%
Assume that $p_0$ is a log canonical but not a quotient singular point. 
Note that the classification of log canonical singularities of dimension two is known (see, e.g., {\cite[p.\ 58]{A92}} or {\cite{Ili86}}), where it is enough to treat only the rational singularities (cf.\ {\cite[Theorem 1.1(2)]{KT09}}). 
Hence, there exists at least one irreducible component $\wGamma _0$ of $\wDelta ^{(0)}$ satisfying $(\wGamma _0 \cdot \wDelta ^{(0)} - \wGamma _0) \ge 3$. 
More precisely, one of the following three subcases holds: 
%%%%%%%%%%%%%%%%%
\begin{enumerate}
\renewcommand{\labelenumi}{(III-\arabic{enumi})}
\item There exists exactly one irreducible component $\wGamma _0$ of $\wDelta ^{(0)}$ satisfying $(\wGamma _0 \cdot \wDelta ^{(0)} - \wGamma _0) =4$. 
\item There exist exactly two irreducible components $\wGamma _{0,1}$ and $\wGamma _{0,2}$ satisfying $(\wGamma _{0,i} \cdot \wDelta ^{(0)} - \wGamma _{0,i})=3$ for $i=1,2$. 
\item There exists exactly one irreducible component $\wGamma _0$ of $\wDelta ^{(0)}$ satisfying $(\wGamma _0 \cdot \wDelta ^{(0)} - \wGamma _0) =3$. 
\end{enumerate}
%%%%%%%%%%%%%%%%%
Notice that Lemma \ref{lem(6-4)} works verbatim for subcases (III-1)--(III-3). 
In what follows, we consider subcases (III-1)--(III-3) separately. 
%%%%%%%%%%%%%%%%%
\begin{rem}[{\cite[3.4.1]{A92}}]\label{rem of lc}
Let $V$ be a normal algebraic surface with a rational log canonical but not quotient singular point $p$ over $\kc$, let $\sigma :\wV \to V$ be the minimal resolution at $p$, and let $E$ be the reduced exceptional divisor of $\sigma$. 
Then $E$ contains at least one $(-m)$-curve with $m>2$. 
Otherwise, by straightforward computing, the determinant of the intersection matrix of $E$ is zero, which contradicts that this intersection matrix is negative-definite ({\cite{Mu61}}). 
\end{rem}
%%%%%%%%%%%%%%%%%
\subsubsection{Subcase (III-1)}\label{6-3-1}
Then the weighted dual graph of $\wDelta ^{(0)}$ is as follows, where $m$ is an integer with $m>2$ (see Remark \ref{rem of lc}): 
\begin{align*}
\xygraph{
\circ ([]!{+(0,-.3)} {^{-m}}) 
((- []!{+(-.8,.4)} \circ ,- []!{+(-.8,-.4)} \circ ), (- []!{+(.8,.4)} \circ ,- []!{+(.8,-.4)} \circ ))
}
\end{align*}
By the similar argument to \S \S \ref{6-2}, $\nu$ is factorized $\nu '' \circ \nu '$ such that $\nu '$ is as in \S \S \ref{6-2}. 
In particular, we can assume that any irreducible component $\wDelta '$ except for $\wGamma _0'$ has self-intersection number $\le -2$, where $\wDelta ' := \nu _{\ast}'(\wDelta _{\kc})$ and $\wGamma _0' := \nu _{\ast}'(\wGamma _0)$. 
By Lemma \ref{lem(6-3)}(3), $\wC _i$ meets a terminal component of $\wDelta ^{(0)}$ for any $i=1,\dots ,n$. 
Moreover, $n \le 4$ by Lemma \ref{lem(6-4)}, in particular, $n \not= 2$ by Lemmas \ref{Morrow} and \ref{lem(3-4)}(1). 
Then we consider two cases of $n=3$ and $n=4$ separately. 

If $n=3$, then $\nu ' = \nu$ and $\sharp \wDelta ' =2$. 
Moreover, $\wGamma _0'$ is a $0$-curve by Lemma \ref{Morrow}(3). 
Therefore, the weighted dual graph of $\wDelta _{\kc}$ is given as $(28)$ or $(29)$ in Appendix \ref{9-3}. 

If $n=4$, then $\nu ' = \nu$ and $\sharp \wDelta ' =1$. 
Moreover, $\wGamma _0'$ is a $1$-curve by Lemma \ref{Morrow}(2). 
Therefore, the weighted dual graph of $\wDelta _{\kc}$ is given as $(30)$ or $(31)$ in Appendix \ref{9-3}. 
%%%%%%%%%%%%%%%%%
\subsubsection{Subcase (III-2)}\label{6-3-2}
Then the weighted dual graph of $\wDelta ^{(0)}$ is as follows, where each $m_i$ is an integer with $m_i \ge 2$ (furthermore, at least one $m_i$ is strictly more than $2$ by Remark \ref{rem of lc}) and $r>1$: 
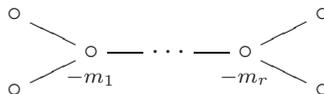
\begin{align*}
\xygraph{
\circ ([]!{+(0,-.3)} {^{-m_1}}) 
((- []!{+(-.8,.4)} \circ,- []!{+(-.8,-.4)} \circ), - []!{+(.8,0)} \cdots - []!{+(.8,0)} \circ ([]!{+(0,-.3)} {^{-m_r}}) (- []!{+(.8,.4)} \circ,- []!{+(.8,-.4)} \circ))
}
\end{align*}
Since $\cDelta$ is a linear chain by Lemma \ref{Morrow}(1), $\nu$ is factorized $\nu '' \circ \nu '$ such that $\wGamma _{0,i}' \not= 0$  for $i=1,2$ and $(\wGamma _{0,1}' \cdot \wDelta '- \wGamma _{0,1}') \le 2$ by replacing $\wGamma _{0,1}$ and $\wGamma _{0,2}$ as needed, where $\wGamma _{0,i}' := \nu _{\ast}'(\wGamma _{0,i})$ for $i=1,2$ and $\wDelta ' := \nu _{\ast}'(\wDelta _{\kc})$. 
For the same reason as in \S \S \ref{6-2}, we can assume that $\nu '$ is defined over $\bk$. 
Then $\nu '$ is uniquely determined, and any irreducible component $\wDelta '$ except for $\wGamma _{0,i}'$ for $i=1,2$ has self-intersection number $\le -2$. 
By noticing $\wGamma _{0,i} \not= 0$ for $i=1,2$ combined with Lemma \ref{lem(6-3)}(3), $\wC _i$ meets a terminal component of $\wDelta ^{(0)}$ for any $i=1,\dots ,n$. 
Moreover, $n \le 4$ by Lemma \ref{lem(6-4)}, in particular, $n \not= 3$ by considering the symmetry of the weighted dual graph of $\wDelta ^{(0)}$. 
In what follows, we consider two cases of $n=2$ and $n=4$ separately. 

In the case of $n=2$, suppose that the weighted dual graph of $\wDelta ^{(0)} + \sum _{i=1}^2\wC _i$ is as follows: 
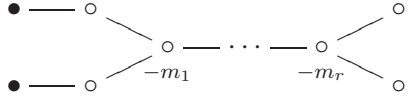
\begin{align*}
\xygraph{
\circ ([]!{+(0,-.3)} {^{-m_1}}) 
((- []!{+(-.8,.4)} \circ - []!{+(-.8,0)} \bullet,- []!{+(-.8,-.4)} \circ - []!{+(-.8,0)} \bullet), - []!{+(.8,0)} \cdots - []!{+(.8,0)} \circ ([]!{+(0,-.3)} {^{-m_r}}) (- []!{+(.8,.4)} \circ,- []!{+(.8,-.4)} \circ))
}
\end{align*}
Then we see that $(\wGamma _{0,1}' \cdot \wDelta ' - \wGamma _{0,1}') =1$ and $\wDelta '$ is not a linear chain, so that $\wGamma _{0,1}'$ is a $(-1)$-curve. 
Moreover, $r > 2$ by Lemma \ref{lem(3-1)}(3). 
Namely, the weighted dual graph of $\wDelta '$ is as follows: 
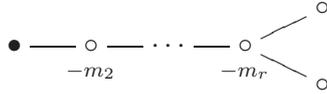
\begin{align*}
\xygraph{
\bullet - []!{+(.8,0)} \circ ([]!{+(0,-.3)} {^{-m_2}}) - []!{+(.8,0)} \cdots - []!{+(.8,0)} \circ ([]!{+(0,-.3)} {^{-m_r}}) (- []!{+(.8,.4)} \circ,- []!{+(.8,-.4)} \circ)
}
\end{align*}
By contracting of $\wDelta ' - \wGamma _{0,1}'$ over $\kc$, we have a log del Pezzo surface of rank one with exactly one quotient singular point of type $D$, which is a contradiction to {\cite[Theorem 3.1(1)]{Ko99}}. 
Hence, the weighted dual graph of $\wDelta ^{(0)} + \sum _{i=1}^2\wC _i$ is as follows: 
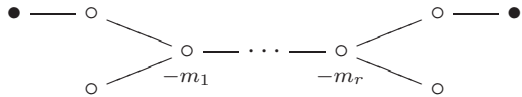
\begin{align*}
\xygraph{
\circ ([]!{+(0,-.3)} {^{-m_1}}) 
((- []!{+(-1,.4)} \circ - []!{+(-.8,0)} \bullet,- []!{+(-1,-.4)} \circ), - []!{+(.8,0)} \cdots - []!{+(.8,0)} \circ ([]!{+(0,-.3)} {^{-m_r}}) (- []!{+(1,.4)} \circ - []!{+(.8,0)} \bullet,- []!{+(1,-.4)} \circ))
}
\end{align*}
Then $\sharp \wDelta ' \ge 4$ and $\wDelta '$ is a linear chain. 
Moreover, $\wGamma _{0,i}'$ is a $(-1)$-curve for $i=1,2$ by noting Lemma \ref{Morrow}(4). 
Thus, the weighted dual graph of $\wDelta '$ satisfies the condition of Proposition \ref{prop(3)}(2) or (3). 
In particular, by Proposition \ref{prop(3)}(2) and (3) and Lemma \ref{lem(6-3)} combined with Remark \ref{rem of lc}, the weighted dual graph of $\wDelta _{\kc}$ is given as $(i)$ $(i=32,33,34)$ in Appendix \ref{9-3}. 

In the case of $n=4$, note that $\wGamma _{0,i}'$ is not a $(-1)$-curve for $i=1,2$ by Lemma \ref{lem(3-1)}(2). 
Hence, we obtain $\nu ' = \nu$. 
Moreover, $\sharp \wDelta ' =2$ and $\wGamma _{0,i}'$ is a $0$-curve for $i=1,2$ by Lemma \ref{Morrow}(3) and (4). 
Therefore, by Lemma \ref{lem(6-3)} combined with Remark \ref{rem of lc}, the weighted dual graph of $\wDelta _{\kc}$ is given as $(35)$ in Appendix \ref{9-3}. 
%%%%%%%%%%%%%%%%%
\subsubsection{Subcase (III-3)}\label{6-3-3}
Then by the similar argument as in \S \S \ref{6-2} combined with the classification of singularities, the weighted dual graph of $\wDelta _{\kc}$ is given as $(i)$ $(i=36,\dots ,52)$ in Appendix \ref{9-3}. \\
%%%%%%%%%%%%%%%%%

The arguments in \S \S \ref{6-1}--\ref{6-3} complete the proof of Theorem \ref{main}(3). 
%%%%%%%%%%%%%%%%%%%%%%%%%%%%%%%%%%%%%%%%%%%%%%%%%%%%%%%%%%%%%%%%%%%%%%%%%%%%%%%%%%%%%%%%%%%%%%%%%%%%%%%%%%
\section{Applications of Theorem \ref{main}(3)}\label{7}
%%%%%%%%%%%%%%%%%%%%%%%%%%%%%%%%%%%%%%%%%%%%%%%%%%%%%%%%
\subsection{Proof of Theorems \ref{main(2)} and \ref{main(3)}}
%%%%%%%%%%%%%%%%%
In this subsection, we shall prove Theorems \ref{main(2)} and \ref{main(3)} by applying Theorem \ref{main}(3). 

Let $S$ be an lc del Pezzo surface of rank one defined over $\bk$ such that $\Sing (S_{\kc}) \not= \emptyset$, and let $\sigma :\wS \to S$ be the minimal resolution over $\bk$. 
Then Theorem \ref{main(2)} can be shown as follows:  
%%%%%%%%%%%%%%%%%
\begin{proof}[Proof of Theorem \ref{main(2)}]
By Theorem \ref{main}(3), we know that (a) implies (b) in Theorem \ref{main(2)}. 
Hence, we shall prove the converse of this. 
Assume that there exists a reduced effective divisor $\wDelta$ on $\wS$ as in Theorem \ref{main(2)}(b). 
By the configuration of the weighted dual graph of $\wDelta$, we can construct the birational morphism $\nu :\wS \to \cS$ over $\bk$ such that $\wS \backslash \Supp (\wDelta ) \simeq \cS \backslash \Supp (\cDelta )$ and the weighted dual graph of $\cDelta _{\kc}$ is either $\xygraph{\circ ([]!{+(0,.2)} {^{1}})}$ or $\xygraph{\circ ([]!{+(0,.2)} {^{0}}) -[l] \circ ([]!{+(0,.2)} {^{m}})}$
$(m \not= -1)$ (see also Example \ref{eg(6-1)}, for an example on the construction of $\nu$), where $\cDelta := \nu _{\ast}(\wDelta )$. 
Meanwhile, letting $\Delta := \sigma _{\ast}(\wDelta )$, we see $\wS \backslash \Supp (\wDelta ) \simeq S \backslash \Supp (\Delta )$ since the exceptional locus of $\sigma$ is included in $\Supp (\wDelta )$. 
Moreover, since $\Delta$ is $\bQ$-ample because of $\rho _{\bk}(S)=1$, we know that $S \backslash \Supp (\Delta )$ is affine by {\cite[Theorem 1]{Go69}}. 
Thus, $\cS _{\kc} \backslash \Supp (\cDelta _{\kc}) \simeq \bA ^2_{\kc}$ by Lemma \ref{Kishimoto}. 
Since there is no non-trivial $\bk$-form of $\bA_{\kc}^2$ ({\cite{Ka75}}), we then have $\cS \backslash \Supp (\cDelta ) \simeq \bA ^2_{\bk}$. 
Therefore, $S \backslash \Supp (\Delta ) \simeq \bA ^2_{\bk}$. 
This implies that the assertion (a) in Theorem \ref{main(2)} holds. 
\end{proof}
%%%%%%%%%%%%%%%%%
\begin{rem}
Notice that it is quite subtle to determine whether $S$ contains the affine plane or not by using only the singularity type on $S_{\kc}$. 
In fact, we can construct some examples that two lc del Pezzo surfaces of rank one with the same singularities such that one contains the affine plane but the other does not (see \S \S \ref{8-3}). 
\end{rem}
%%%%%%%%%%%%%%%%%
\begin{eg}\label{eg(6-1)}
With the notation as above, we shall consider a case that there exists a reduced effective divisor $\wDelta$ on $\wS$ such that the exceptional locus of $\sigma$ is included in $\Supp (\wDelta )$ and the weighted dual graph of $\wDelta$ is as (21) in Appendix \ref{9}. 
Let $\wC _1,\dots ,\wC_8$ be all irreducible components of $\wDelta _{\kc}$ named as follows: 
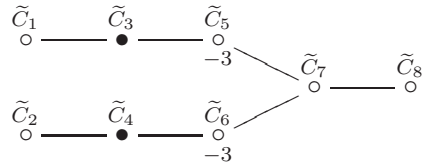
\begin{align*}
\xygraph{
\circ ([]!{+(0,.2)} {^{\wC _8}}) -[l] \circ ([]!{+(0,.2)} {^{\wC _7}})
(- []!{+(-1,.5)} \circ ([]!{+(0,.2)} {^{\wC _5}}) ([]!{+(0,-.25)} {^{-3}}) -[l] \bullet ([]!{+(0,.2)} {^{\wC _3}}) -[l] \circ ([]!{+(0,.2)} {^{\wC _1}}), - []!{+(-1,-.5)} \circ ([]!{+(0,-.25)} {^{-3}}) ([]!{+(0,.2)} {^{\wC _6}}) -[l] \bullet ([]!{+(0,.2)} {^{\wC _4}}) -[l] \circ ([]!{+(0,.2)} {^{\wC _2}}))}
\end{align*}
By the symmetry of this graph, $\wC _1+\wC _2$, $\wC _3+\wC _4$, $\wC _5+\wC _6$, $\wC _7$ and $\wC _8$ are defined over $\bk$. 
Let $\nu :\wS \to \cS$ be the compositions of successive contractions of a disjoint union $\wC _3+\wC _4$, that of the images of $\wC _1+\wC _2$ and finally that of the images of $\wC_5 + \wC _6$. 
By construction, $\nu$ is defined over $\bk$.  
Moreover, putting $\cDelta := \nu _{\ast}(\wDelta )$, then $\cDelta _{\kc}$ consists of two irreducible components $\cC _7 := \nu _{\ast}(\wC _7)$ and $\cC _8 := \nu _{\ast}(\wC _8)$ such that $\cC _7$ and $\cC _8$ are a $0$-curve and a $(-2)$-curve, respectively. 
Thus, $\cS \simeq \bF _2$ and $\cS \backslash (\cC _7 \cup \cC _8) \simeq \bA ^2_{\bk}$. 
In particular, $S \backslash \Supp (\sigma _{\ast}(\wDelta )) \simeq \bA ^2_{\bk}$. 
\end{eg}
%%%%%%%%%%%%%%%%%
From now on, we shall prove Theorem \ref{main(3)}. 
Assume that $S$ has at most Du Val singularities, and let $d$ be the degree of $S$, i.e., $d := (-K_S)^2$. 
Then we will use the following fact: 
%%%%%%%%%%%%%%%%%
\begin{lem}\label{degree}
Let the notation and the assumptions be the same as above. 
If $d \ge 5$, then the pair of the degree and singularity type of $S_{\kc}$ is $(8,A_1)$, $(6,A_2+A_1)$, $(6,A_2)$, $(6,A_1)$ (with $3$ lines) or $(5,A_4)$. 
\end{lem}
%%%%%%%%%%%%%%%%%
\begin{proof}
This assertion follows from the argument in {\cite[\S \S 3.1]{Sw}}. 
\end{proof}
%%%%%%%%%%%%%%%%%
\begin{proof}[Proof of Theorem \ref{main(3)}]
We shall consider the two cases of $\rho _{\kc}(S_{\kc})=1$ or $\rho _{\kc}(S_{\kc})>1$ separately. 

In the case of $\rho _{\kc}(S_{\kc})=1$, then looking for all weighted dual graphs in {\cite[Appendix C]{Ko01}} such that each vertex corresponds to either a $(-1)$-curve or a $(-2)$-curve, we know that such the graphs are summarized in (1), (14), (2), (3), (5), (7) and (12) in {\cite[Appendix C]{Ko01}}, where we assume $n=2$ for graphs (1), (14), (2) and (3), and that the subgraph $A$ consists of only one vertex corresponding to a $(-2)$-curve for graphs (14) and (2). 
Notice that the these graphs correspond to the pair of the degree and singularity type of $S _{\kc}$ $(8,A_1)$, $(6,A_2+A_1)$, $(5,A_4)$, $(4,D_5)$, $(3,E_6)$, $(2,E_7)$ and $(1,E_8)$. 
Moreover, for each graph except for (1), the $(-1)$-curve corresponding to the vertex in this graph is always defined over $\bk$. 
Meanwhile, for graph (1), there exists a curve corresponding to the vertex with the weight zero defined over $\bk$ if and only if $S$ has a smooth $\bk$-rational point. 

In the case of $\rho _{\kc}(S_{\kc})>1$, then looking for all weighted dual graphs in Appendix \ref{9} such that each vertex corresponds to either a $(-1)$-curve or a $(-2)$-curve, we know that such the graphs are summarized in (1), (2), (4), (5), (8), (18), (24) and (26) in Appendix \ref{9}, where we assume $(t,n)=(0,3)$ (resp. $(t,n,m)=(0,2,2)$, $(t,n')=(0,2)$, $(t,n')=(0,1)$, $m=2$) for graphs (1) (resp. (2), (4), (5), (18)) and that the subgraph $A$ consists of only one vertex corresponding to a $(-2)$-curve for graphs (5) and (8). 
Notice that the these graphs correspond to the pair of the degree and singularity type of $S _{\kc}$ $(6,A_1)$ (with $3$ lines), $(6,A_2)$, $(4,A_2)$, $(2,A_6)$, $(4,A_2+2A_1)$, $(4,D_4)$, $(2,E_6)$ and $(3,D_4)$. 
Moreover, for each graph, the union of $(-1)$-curves corresponding to all vertices $\bullet$ is always defined over $\bk$. 

Noticing Lemma \ref{degree}, by using Theorem \ref{main(2)} this completes the proof. 
\end{proof}
%%%%%%%%%%%%%%%%%%%%%%%%%%%%%%%%%%%%%%%%%%%%%%%%%%%%%%%%
\subsection{Application to singular del Pezzo fibrations}
%%%%%%%%%%%%%%%%%
Let $f:X \to Y$ be a dominant and projective morphism between normal varieties defined over $\bC$ such that the generic fiber $X_{\eta}$ of $f$ is a numerical del Pezzo surface of rank one. 
We say that $f$ is a {\it canonical} (resp. an {\it lt}, an {\it lc}) {\it del Pezzo fibration} if $X$ has at most canonical (resp. log terminal, log canonical) singularities. 
By {\cite[Lemma 3]{DK18}}, $f$ admits a vertical $\bA ^2_{\bC}$-cylinder (see {\cite{DK18}}, for this definition) if and only if $X_{\eta}$ contains $\bA ^2_{\bC (Y)}$. 
We shall consider the existing condition of vertical $\bA ^2_{\bC}$-cylinders with respect to $f$. 
%%%%%%%%%%%%%%%%%
Assume that $f$ is a canonical del Pezzo fibration such that $\Sing (X_{\eta ,\overline{\bC (Y)}}) \not= \emptyset$. 
We say that $(-K_{X_{\eta}})^2 \in \{ 1,\dots ,6,8\}$ is the {\it degree} of $f$. 
By Theorem \ref{main(3)} combined with {\cite[Lemma 3]{DK18}}, we have the following: 
%%%%%%%%%%%%%%%%%
\begin{cor}\label{cano}
Let $f:X \to Y$ be a canonical del Pezzo fibration of degree $d$ and let $X_{\eta}$ be the generic fiber of $f$ such that $\Sing (X_{\eta ,\overline{\bC (Y)}}) \not= \emptyset$. 
Then we have the following: 
\begin{enumerate}
\item If $d=8$, then $f$ admits a vertical $\bA ^2_{\bC}$-cylinder if and only if $X_{\eta}$ has a smooth $\bC (Y)$-rational point. 
\item If $d=5$ or $6$, then $f$ always admits a vertical $\bA ^2_{\bC}$-cylinder. 
\item If $d \le 4$, $f$ admits a vertical $\bA ^2_{\bC}$-cylinder if and only if the pair of the degree $d$ and the singularity type of $X_{\eta ,\overline{\bC (Y)}}$ is one of the following: 
\begin{align*}
(4,D_5),\,(4,D_4),\,(4,A_2+2A_1),\,(4,A_2),\,(3,E_6),\,(3,D_4), \, (2,E_7),\,(2,E_6),\,(2,A_6),\,(1,E_8). 
\end{align*}
\end{enumerate}
\end{cor}
%%%%%%%%%%%%%%%%%
\begin{rem}
If a canonical del Pezzo fibration $f:X \to Y$ of degree $d$, whose generic fiber $X_{\eta}$ of $f$ satisfies $\Sing (X_{\eta ,\overline{\bC (Y)}}) = \emptyset$, then by {\cite{DK18}} we know that $f$ admits a vertical $\bA ^2_{\bC}$-cylinder if and only if $d \ge 8$ and $X_{\eta}$ has a $\bC (Y)$-rational point. 
\end{rem}
%%%%%%%%%%%%%%%%%
\begin{eg}
Let $\sO$ be a discrete valuation ring of $\bC (t)$ such that the maximal ideal of $\sO$ is generated by $t$ and let $X$ be the three-dimensional algebraic variety over $\bC$ defined by: 
\begin{align*}
X := (tw^2+xy^3+z^4+yzw=0) \subseteq \bP _{\sO}(1,1,1,2) = \Proj (\sO [x,y,z,w]). 
\end{align*}
Let $f:X \to \Spec (\sO )$ be the structure morphism as $\sO$-scheme and let $\eta$ be the generic point of $\Spec (\sO )$. 
Then the generic fiber $X_{\eta}$ of $f$ is an irreducible quartic hypersurface of the weighted projective space given by: 
\begin{align*}
X_{\eta} = (tw^2+xy^3+z^4+yzw=0) \subseteq \bP _{\bC (t)}(1,1,1,2) = \Proj (\bC (t)[x,y,z,w]). 
\end{align*}
Then $X_{\eta ,\overline{\bC (t)}}$ is a Du Val del Pezzo surface of degree $2$ with exactly one singular point $p := [1\!:0\!:\!0\!:\!0]$ of type $E_6$. 
Note that the weighted dual graph of all $(-1)$-curves and $(-2)$-curves on the minimal resolution of $X_{\eta ,\overline{\bC (Y)}}$ is as $(11^{o})$ in {\cite[p. 349]{CP21}}. 
Hence, we see that $f$ is a canonical del Pezzo fibration of degree $2$. 
Moreover, it admits a vertical $\bA ^2 _{\bC}$-cylinder by Corollary \ref{cano}. 
\end{eg}
%%%%%%%%%%%%%%%%%
Notice that Theorem \ref{main(2)} provides a way to determine whether lt del Pezzo fibrations and lc del Pezzo fibrations admit vertical $\bA ^2_{\bC}$-cylinders or not. 
%%%%%%%%%%%%%%%%%%%%%%%%%%%%%%%%%%%%%%%%%%%%%%%%%%%%%%%%%%%%%%%%%%%%%%%%%%%%%%%%%%%%%%%%%%%%%%%%%%%%%%%%%%
\section{Remarks on Theorem \ref{main}}\label{8}
%%%%%%%%%%%%%%%%%%%%%%%%%%%%%%%%%%%%%%%%%%%%%%%%%%%%%%%%
\subsection{Existence of compactifications of the affine plane}
%%%%%%%%%%%%%%%%%
In this subsection, we shall discuss whether there exists indeed a compactification of the affine plane into an lc del Pezzo surface of rank one corresponding to the weighted dual graph $(i)$ in Appendix \ref{9} for each $i=1,\dots ,52$. 

We shall prepare the following condition (\ref{condition}) with respect to the base field $\bk$: 
%%%%%%%%%%%%%%%%%
\begin{align*}
\label{condition}\tag{$\ast$}
\text{For any $n \in \bZ _{>0}$,\ } \text{\ there exist two elements in\ } \kc\text{, which are not Galois conjugate over\ } \bk \text{, }\\ \text{\ such that their minimal polynomials over\ } \bk \text{\ are of degree\ } 2 \text{\ and of degree\ } n \text{, respectively. }
\end{align*}
%%%%%%%%%%%%%%%%%
For instance, the rational number field $\bQ$ and any rational function field satisfy the above condition (\ref{condition}). 
Meanwhile, the real number field $\bR$ does not satisfy (\ref{condition}). 

Letting $G$ be one of the weighted dual graphs in Appendix \ref{9}, assume that there exists a compactification $(S,\Delta )$ of the affine plane $\bA ^2_{\bk}$ into an lc del Pezzo surface $S$ of rank one over $\bk$ corresponding to this graph $G$. 
More precisely, letting $\sigma :\wS \to S$ be the minimal resolution and letting $\wDelta := \sigma ^{\ast}(\Delta )_{\red}$, then the weighted dual graph of $\wDelta _{\kc}$ is the same as $G$. 
By Lemmas \ref{lem(2-2)} and \ref{not mnc}, we notice that all $(-1)$-curves, which are included in $\Supp (\wDelta _{\kc})$, lie in the same $\Gal$-orbit. 
Hence, we can completely see the configuration of $\Gal$-orbits of each irreducible component of $\wDelta _{\kc}$. 
More precisely, one of the following four situations occurs:  
%%%%%%%%%%%%%%%%%

{\it Situation 1:} 
There exist connecting two vertices corresponding to two curves, which lie in the same $\Gal$-orbit. 
Moreover, there exist exactly two vertices $v^{(1)}$ and $v^{(2)}$ such that they are respectively connected to $n_2$-times of vertices, in which $2n_2$-times of curves corresponding to these vertices lie in the same $\Gal$-orbit, where $n_2 \ge 2$. 
Thus, by the configuration of $G$ there exist two elements in $\kc$, which are not Galois conjugate over $\bk$, such that their minimal polynomials over $\bk$ are of degree $2$ and of degree $n_2$, respectively. 
This situation only occurs when the graph $G$ is one of the following graphs in Appendix \ref{9}: 
\begin{itemize}
\item (4), (5), where $n_2:=n'$ with $n' \ge 2$; 
\item (35), where $n_2:=2$. 
\end{itemize}
%%%%%%%%%%%%%%%%%

{\it Situation 2:} 
There exist connecting two vertices corresponding to two curves, which lie in the same $\Gal$-orbit. 
However, there are no vertices $v^{(1)}$ and $v^{(2)}$ as in Situation 1. 
Thus, by the configuration of $G$ there exists an element in $\kc$ such that its minimal polynomial over $\bk$ is of degree $2$. 
This situation occurs when the graph $G$ is one of the graphs (5) with $n'=1$, (8) and (32) in Appendix \ref{9}. 
%%%%%%%%%%%%%%%%%

{\it Situation 3:} 
There exists a unique vertex, which corresponds to a curve defined over $\bk$, such that it is connected to exactly $n_1$-times of vertices corresponding to curves, which lie in the same $\Gal$-orbit, where $n_1 \ge 2$. 
Moreover, there exist exactly $n_1$-times of vertices $v^{(1)},\dots,v^{(n_1)}$ such that they are respectively connected to $n_2$-times of vertices, in which $n_1n_2$-times of curves corresponding to these vertices lie in the same $\Gal$-orbit, where $n_2 \ge 2$. 
Thus, by the configuration of $G$ there exist two elements in $\kc$, which are not Galois conjugate over $\bk$, such that their minimal polynomials over $\bk$ are of degree $n_1$ and of degree $n_2$, respectively. 
This situation only occurs when the graph $G$ is one of the following graphs in Appendix \ref{9}: 
\begin{itemize}
\item (6), (7), where $(n_1,n_2):=(2,n')$ with $n' \ge 2$; 
\item (20), (35), (45), where  $(n_1,n_2):=(2,2)$; 
\item (37), where $(n_1,n_2):=(2,3)$; 
\item (48), where $(n_1,n_2):=(3,2)$. 
\end{itemize}
%%%%%%%%%%%%%%%%%

{\it Situation 4:} 
There exists a unique vertex, which corresponds to a curve defined over $\bk$ such that it is connected to exactly $n_1$-times of vertices corresponding to curves, which lie in the same $\Gal$-orbit, where $n_1 \ge 2$. 
However, there are no vertices $v^{(1)},\dots,v^{(n_1)}$ as in Situation 3. 
Thus, by the configuration of $G$ there exists an element in $\kc$ such that its minimal polynomial over $\bk$ is of degree $n_1$. 
This situation only occurs when the graph $G$ is one of the following graphs in Appendix \ref{9}: 
\begin{itemize}
\item (1), where $n_1:= n$ with $n \ge 3$; 
\item (2), (3), where $n_1:= n$ with $n \ge 2$; 
\item (36), where $n_1:= 5$; 
\item (14), (30), (31), where $n_1:= 4$; 
\item (13), (26), (27), (28), (29), (49), (50), (51), (52), where $n_1:=3$; 
\item Otherwise, where $n_1:=2$. 
\end{itemize}
%%%%%%%%%%%%%%%%%

In summary, if for any weighted dual graph in Appendix \ref{9} there exists a compactification of the affine plane over $\bk$ into an lc del Pezzo surface of rank one corresponding to this graph, then the base field $\bk$ must satisfy the condition (\ref{condition}). 
%%%%%%%%%%%%%%%%%
\begin{eg}
Let $(S,\Delta )$ be a comapctification of the affine plane $\bA ^2_{\bR}$ over the real number field $\bR$, let $\sigma :\wS \to S$ be the minimal resolution over $\bR$ and let us put $\wDelta := \sigma ^{\ast}(\Delta )_{\red}$. 
Assume that $S$ is an lc del Pezzo surface of rank one over $\bR$ such that $\Sing (S_{\bC})\not= \emptyset$. 
Notice that $\bR$ does not satisfy the condition (\ref{condition}) because of ${\rm Gal}(\bC /\bR ) \simeq \bZ /2\bZ$. 
In particular, the weighted dual graph of $\wDelta _{\bC}$ does not appear in the list of Situations 1 or 3. 
If the weighted dual graph of $\wDelta _{\bC}$ occurs in Situation 2, then this graph is one of (5) with $n'=1$, (8) and (32) in Appendix \ref{9}.
If the weighted dual graph of $\wDelta _{\bC}$ occurs in Situation 4, then this graph is one of (5) with $n'=1$, (2)--(3) with $n=2$, (6)--(7) with $n'=1$, (9)--(12), (15)--(19), (21)--(25), (33)--(34), (38)--(44), (46) and (47) in Appendix \ref{9}. 
\end{eg}
%%%%%%%%%%%%%%%%%
Conversely, assuming that $\bk$ satisfies the condition (\ref{condition}), let $G$ be one of the weighted dual graphs in Appendix \ref{9}. 
Then we can explicitly construct a compactification $(S,\Delta )$ of $\bA ^2_{\bk}$ into an lc del Pezzo surface $S$ of rank one over $\bk$. 
We explain the method of this construction. 
Let $(\cS ,\cDelta )$ be the mnc of $\bA ^2_{\bk}$ over $\bk$ according to the configuration of $G$ as follows: 
%%%%%%%%%%%%%%%%%
\begin{itemize}
\item If $G$ is as the graph $(1)$, $(6)$, $(7)$, $(9)$, $(10)$, $(26)$, $(27)$, $(30)$, $(31)$, $(33)$, $(34)$, $(48)$, $(49)$, $(50)$, $(51)$ or $(52)$ in Appendix \ref{9}, then $(\cS ,\cDelta) := (\bP ^2_{\bk},L)$, where $L$ is a general line on $\cS \simeq \bP ^2_{\bk}$; 
\item If $G$ is as the graph $(4)$, $(5)$, $(8)$, $(32)$ or $(35)$ in Appendix \ref{9}, then $\cS$ is a $\bk$-form of $\bP ^1_{\kc} \times \bP ^1_{\kc}$ of rank one and $\cDelta := F_1+F_2$, where $F_1$ and $F_2$ are $\bk$-forms of irreducible curves of type $(1,0)$ and $(0,1)$, respectively. Notice that $\cDelta$ is defined over $\bk$; 
\item If $G$ is as the graph $(11)$, $(13)$, $(14)$, $(15)$, $(16)$, $(17)$, $(20)$, $(21)$, $(22)$, $(23)$, $(24)$, $(25)$, $(28)$, $(29)$, $(36)$, $(37)$, $(38)$, $(39)$, $(43)$ or $(44)$ in Appendix \ref{9}, then $(\cS ,\cDelta ) := (\bF _2,M+F)$, where $M$ and $F$ are the minimal section and a general fiber of the structure morphism $\bF _2 \to \bP ^1_{\bk}$ over $\bk$, respectively; 
\item If $G$ is as the graph $(12)$, $(40)$, $(41)$, $(42)$, $(45)$, $(46)$ or $(47)$ in Appendix \ref{9}, then $(\cS ,\cDelta ) := (\bF _3,M+F)$, where $M$ and $F$ are the minimal section and a general fiber of the structure morphism $\bF _3 \to \bP ^1_{\bk}$ over $\bk$, respectively; 
\item If $G$ is as the graph $(2)$, $(3)$, $(18)$ or $(19)$ in Appendix \ref{9}, then $(\cS ,\cDelta) := (\bF _m ,M+F)$, where $M$ and $F$ are the minimal section and a general fiber of the structure morphism $\bF _m \to \bP ^1_{\bk}$ over $\bk$, respectively. 
\end{itemize}
%%%%%%%%%%%%%%%%%
Then we can construct two birational morphisms $\nu :\wS \to \cS$ and $\sigma :\wS \to S$ over $\bk$ such that the weighted dual graph of $\wDelta _{\kc}$ is the same as $G$ and $(S,\Delta)$ is a compactification of $\bA ^2_{\bk}$ into an lc del Pezzo surface $S$ of rank one over $\bk$, where $\wDelta := \nu ^{\ast}(\cDelta )_{\red}$ and $\Delta := \sigma _{\ast}(\wDelta )$. 
For specific construction of the above two birational morphisms, see the following example (notice that we can construct in a similar way for other cases): 
%%%%%%%%%%%%%%%%%
\begin{eg}
Assume that the base field $\bk$ satisfies (\ref{condition}), and let $G$ be the weighted dual graph of $\wDelta$ is as (21) in Appendix \ref{9}. 
Since (\ref{condition}) holds, there exist two elements $a,b \in \kc$, which are not Galois conjugate over $\bk$, such that there are exactly two (resp. three) elements $a_1,a_2 \in \kc$ (resp. $b_1,b_2,b_3 \in \kc$), which are Galois conjugates of $a$ (resp. $b$) over $\bk$, where $a_1 := a$ and $b_1 := b$. 
Let $P(t) \in \bk [t]$ be the minimal polynomial for $a$ over $\bk$. 
Now, put $\cS := \bF _2$, and let $F$ and $M$ be a fiber and the minimal section of the structure morphism $\cS \simeq \bF _2 \to \bP ^1_{\bk}$ over $\bk$, respectively. 
Then we shall take an affine open neighborhood $U \simeq \Spec (\bk [x,y])$ such that $\ell := F \cap U \simeq (x=0) \subseteq \bA ^2_{\bk}$. 
Let $\nu ' : \cS ' \to \cS _{\kc}$ be the blow-up at two points $(0,a_i) \in \bA ^2_{\kc}$ for $i=1,2$. 
Note that $\nu '$ is defined over $\bk$. 
Then the pullback ${\nu '}^{-1}(\ell )$ and the exceptional set $E$ of $\nu '$ can be written by $(u=0)$ and $(P(y)=0)$ in $\bA ^1_{\bk} \times \bP ^1_{\bk} = \Spec (\bk [y]) \times \Proj (\bk [u,v])$, respectively. 
Hence, we can construct the blow-up $\nu '' :\wS \to \cS ' _{\kc}$ at six points $a_i \times [1\! :\! b_j] \in \bA ^1_{\kc} \times \bP ^1_{\kc}$ for $i=1,2$ and $j=1,2,3$. 
Noticing $\nu ''$ is defined over $\bk$, so is $\nu := \nu ' \circ \nu ''$. 
Now, let $\wE$ be the reduced exceptional divisor of $\nu$, and put $\widetilde{F} := \nu ^{-1}_{\ast}(F)$ and $\widetilde{M} := \nu ^{-1}_{\ast}(M)$. 
Then the weighted dual graph of the reduced divisor $\wDelta := \wE + \widetilde{F} + \widetilde{M}$ on $\wS$ is as in $G$. 
Moreover, we know that ${\nu ''}^{-1}_{\ast}(E) + \widetilde{F} + \widetilde{M}$ can be contracted, hence, we obtain this contraction $\sigma :\wS \to S$ over $\bk$. 
By construction, letting $\Delta := \sigma _{\ast}(\wDelta )$, we see that $(S,\Delta )$ is certainly a compactification of $\bA ^2_{\bk}$ such that $S$ is an lc del Pezzo surfaces of rank one over $\bk$. 
\end{eg}
%%%%%%%%%%%%%%%%%%%%%%%%%%%%%%%%%%%%%%%%%%%%%%%%%%%%%%%%
\subsection{Maximal number of singular points}
Let $(S,\Delta )$ be a compactification of the affine plane $\bA ^2_{\bk}$ over $\bk$. 
Assume that $S$ is an lc del Pezzo surface of rank one defined over $\bk$ such that $\Sing (S_{\kc}) \not= \emptyset$. 
By Theorem \ref{main}(2)(iii), we obtain $\sharp \Sing (S_{\kc}) \le \rho _{\kc}(S_{\kc})+1$, which can be regarded as a generalization of the case of $\bk = \kc$ (see Lemma \ref{lem(2-3)}). 
In particular, assuming that $\bk$ satisfies (\ref{condition}), for any positive integer $n$, there exists a log del Pezzo surface $S_n$ of rank one defined over $\bk$ containing $\bA ^2_{\bk}$ such that $\sharp \Sing (S_{n,\kc})=n+1$. 
Indeed, it follows from {\cite{Ko01}} (resp. the weighted dual graph (2) in Appendix \ref{9}) if $n=1$ (resp. $n \ge 2$). 
Meanwhile, we see $\sharp \Sing (S_{\kc}) \le 4$ (resp. $\sharp \Sing (S_{\kc}) \le 5$) if $S_{\kc}$ has a non-cyclic quotient singular point (resp. log canonical but not a quotient singular point) by Theorem \ref{main}(3) (see also Appendix \ref{9}). 
%%%%%%%%%%%%%%%%%%%%%%%%%%%%%%%%%%%%%%%%%%%%%%%%%%%%%%%%
\subsection{Converse of Theorem \ref{main}(3)}\label{8-3}
Let $S$ be an lc del Pezzo surface of rank one over $\bk$. 
If $S$ contains $\bA ^2_{\bk}$, we classify this boundary divisor by Theorem \ref{main}(3). 
Now, we shall consider this converse. 
According to Theorem \ref{main(2)}, if there exists a reduced divisor $\wDelta$ on $\wS$ such that the exceptional set of $\sigma$ is included in $\Supp (\wDelta )$, each irreducible component of $\wDelta _{\kc}$ is a rational curve and the weighted dual graph of $\wDelta _{\kc}$ is one of the graphs in {\cite[Appendix C]{Ko01}}, {\cite[Fig. 1]{KT09}} or Appendix \ref{9}, then $S$ contains $\bA ^2_{\bk}$. 
Here, notice that $\wDelta$ certainly depends on the singularity type of $S_{\kc}$. 
Hence, we shall further consider the following problem: 
%%%%%%%%%%%%%%%%%
\begin{prob}[cf.\ {\cite[Problem 1]{Ko01}}]\label{converse}
Let $S$ be an lc del Pezzo surface of rank one over $\bk$. 
Assume that the singularity type of $S_{\kc}$ is given as one of the graphs in {\cite[Appendix C]{Ko01}}, {\cite[Fig. 1]{KT09}} or Appendix \ref{9}. 
Does then $S$ contain the affine plane $\bA ^2_{\bk}$?
\end{prob}
%%%%%%%%%%%%%%%%%
In the case of $\bk = \kc$, Problem \ref{converse} is not true in general ({\cite[\S 4]{Ko01}). 
However, Problem \ref{converse} is true if $S$ has a singular point, which is not a cyclic quotient singularity, 
({\cite{KT09,KT18}}) or $S$ is a log del Pezzo surface of index $\le 3$ ({\cite{MZ88,Ko03,Ko01}}). 
On the other hand, in the case of $\bk \not= \kc$, we find some counter-examples of Problem \ref{converse} as follows: 
%%%%%%%%%%%%%%%%%
\begin{eg}\label{E6}
Let $S$ be a Du Val del Pezzo surface of rank one, and let $\sigma :\wS \to S$ be the minimal resolution. 
Assume that $S_{\kc}$ has only one singular point $p$ of type $E_6$ and $\rho _{\bk}(\wS ) - \rho _{\bk}(S)=4$ (i.e., following the notation of {\cite{Sw}}, $p$ is of type $E_6^+$ on $S$). 
Then the degree $d := (-K_S)^2$ of $S$ is equal to $1$ or $2$. 
If $d=2$, then $S$ contains $\bA ^2_{\bk}$ since $\wS _{\kc}$ contains a reduced effective divisor with a weighted dual graph as (24) in Appendix \ref{9}. 
If $d=1$, then $S$ does not contain $\bA ^2_{\bk}$ since it does not contain any cylinder by {\cite[Theorem 1.6]{Sw}}. 
\end{eg}
%%%%%%%%%%%%%%%%%
\begin{eg}
Assume that there exist two elements in $\kc$ such that their minimal polynomials over $\bk$ are of degree $2$ and $4$, respectively. 
Let us fix the Hirzebruch surface $\bF _3$ of degree $3$ defined over $\bk$ and let $M$ be the minimal section of the structure morphism $\pi :\bF _3 \to \bP ^1_{\bk}$. 
Let $F_1,\dots ,F_4$ be four fibers of $\pi _{\kc}$, which lie in the same $\Gal$-orbit, and let $\{ x_{i,j}\} _{1 \le i \le 4,\, 1 \le j \le 2}$ be eight points on $\bF _3$ such that $x_{i,1}$ and $x_{i,2}$ lie in both the fiber $F_i$ of $\pi$ and the same $\Gal$-orbit for $i=1,\dots ,4$, and $\sum _{i=1}^4\sum_{j=1}^2x_{i,j}$ is defined over $\bk$. 
Note that we can certainly take these points by the assumption. 
Letting $\nu :\wS \to \bF _3$ be a blow-up at $\{ x_{i,j}\} _{1 \le i \le 4,\, 1 \le j \le 2}$, the weighted dual graph of $\nu ^{\ast}(M+F_1+\dots +F_4)_{\red}$ is as follows: 
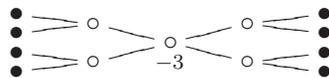
\begin{align*}
\xygraph{
\circ ([]!{+(0,-.25)} {^{-3}}) 
((- []!{+(-.8,.2)} \circ (- []!{+(-.8,.1)} \bullet ,- []!{+(-.8,-.1)} \bullet ),- []!{+(-.8,-.2)} \circ (- []!{+(-.8,.1)} \bullet ,- []!{+(-.8,-.1)} \bullet )), (- []!{+(.8,.2)} \circ (- []!{+(.8,.1)} \bullet ,- []!{+(.8,-.1)} \bullet ),- []!{+(.8,-.2)} \circ (- []!{+(.8,.1)} \bullet ,- []!{+(.8,-.1)} \bullet )))}
\end{align*}

Let $\sigma :\wS \to S$ be the contraction of $\nu ^{-1}_{\ast}(M+F_1+\dots +F_4)$. 
By construction, $S$ is then an lc del Pezzo surface of rank one, which has a log canonical but not a quotient singular point. 
In particular, the singularity type of $S_{\kc}$ is the same singularity type as (30) in Appendix \ref{9}, however, $S$ does not contain $\bA ^2_{\bk}$ by Theorem \ref{main}(3). 
\end{eg}
%%%%%%%%%%%%%%%%%
\begin{eg}
Let $m$ be a positive integer, let $C$ be the plane conic over $\bQ$ defined by $(xz=y^2) \subseteq \bP ^2_{\bQ} = \Proj (\bQ [x,y,z])$ and let $x_1,\dots ,x_{2m+4}$ be points on $C_{\overline{\bQ}}$ given by $x_i := [1\!:\!\sqrt[2m+4]{2} \zeta ^i\!:\! \sqrt[m+2]{2}\zeta ^{2i}] \in \bP ^2_{\overline{\bQ}}$ for $i=1,\dots ,2m+4$, where $\zeta := \exp \left( \frac{\pi \sqrt{-1}}{m+2}\right)$. 
Then $\sum _{i=1}^{2m+4}x_i$ is defined over $\bQ$. 
Let $\nu :\wS \to \bP ^2_{\bQ}$ be a blow-up at $x_1,\dots ,x_{2m+4}$ over $\bQ$ and let $\sigma :\wS \to S$ be the contraction of $\nu ^{-1}_{\ast}(C)$, which is a $\bQ$-form of $(-2m)$-curve. Then $S$ is a log del Pezzo surface of rank one over $\bQ$. 
Since $S_{\overline{\bQ}}$ has exactly one singular point, whose singularity type is the same as the singular point on $\bP _{\overline{\bQ}}(1,1,2m)$ over $\overline{\bQ}$, we see that $S$ is of index $m$. 
space $\bP _{\overline{\bQ}}(1,1,2m)$ over $\overline{\bQ}$, we see that $S$ is of index $m$. 
However, $S$ does not contain $\bA ^2_{\bQ}$ by Theorem \ref{main}(3). 
Meanwhile, the weighted projective space $\bP _{\bQ}(1,1,2m)$ is a log del Pezzo surface of rank one and of index $m$, and contains $\bA ^2_{\bQ}$. 
\end{eg}
%%%%%%%%%%%%%%%%%%%%%%%%%%%%%%%%%%%%%%%%%%%%%%%%%%%%%%%%%%%%%%%%%%%%%%%%%%%%%%%%%%%%%%%%%%%%%%%%%%%%%%%%%%
\appendix
%%%%%%%%%%%%%%%%%%%%%%%%%%%%%%%%%%%%%%%%%%%%%%%%%%%%%%%%%%%%%%%%%%%%%%%%%%%%%%%%%%%%%%%%%%%%%%%%%%%%%%%%%%
\section{List of configurations}\label{9}
%%%%%%%%%%%%%%%%%
Letting the notation and the assumptions be the same as in Theorem \ref{main}(3), 
we summarize configurations of all weighted dual graphs of $\wDelta _{\kc}$, where we employ the following notation: 
%%%%%%%%%%%%%%%%%
\begin{itemize}
\item For the following all weighted dual graphs, $t$, $t'$ and $m$ are arbitrary integers with $t \ge 0$, $t' \ge 0$ and $m \ge 2$. 
\item In (4), (5), (6) or (7), assume that $n$ is even and let $n'$ be the integer with $2n'=n$. 
\item In (3), (5), (7), (8) or (10), the subgraph $A$ means an arbitrary admissible twig, and $m_{A}$ means the integer as in Definition \ref{mA}. Moreover, the subgraphs ${}^tA$ and ${A}^*$ denote the transposal and adjoint of $A$, respectively (see Definition \ref{def(3-1)}). 
Meanwhile, if $A$ can be written by $\xygraph{\circ ([]!{+(0,-.3)} {^{-m_1}}) - []!{+(.8,0)} \circ ([]!{+(0,-.3)} {^{-m_2}}) - []!{+(.8,0)} \cdots - []!{+(.8,0)} \circ ([]!{+(0,-.3)} {^{-m_r}})}$, the subgraph $\underline{A}$ denotes $\xygraph{\circ ([]!{+(0,-.3)} {^{-m_1}}) - []!{+(.8,0)} \circ ([]!{+(0,-.3)} {^{-m_2}}) - []!{+(.8,0)} \cdots - []!{+(.8,0)} \circ ([]!{+(0,-.3)} {^{-m_{r-1}}})}$. 
\item The subgraphs $U(t)$ and $L (m;t)$ mean 
$\left. \begin{array}{l c}
\ldelim\{{5}{34pt}[{\scriptsize $t$-vertices}] & \xygraph{\circ} \\
& {\xygraph{(- []!{+(0,-.1)} ,- []!{+(0,.1)} )}} \\
& \xygraph{\vdots} \\
& {\xygraph{(- []!{+(0,-.1)} ,- []!{+(0,.1)} )}} \\
%\ldelim\{{3.4}{34pt}[{\scriptsize $t$-vertices}] & \xygraph{\circ} \\[-0.7zh]
%& {\xygraph{(- []!{+(0,-.1)} ,- []!{+(0,.1)} )}} \\[-0.85zh]
%& \xygraph{\vdots} \\[-0.35zh]
%& {\xygraph{(- []!{+(0,-.1)} ,- []!{+(0,.1)} )}} \\[-0.6zh]
& \xygraph{\circ} 
\end{array} \right. $ and $\xygraph{\circ ([]!{+(0,-.3)} {^{-m}}) - []!{+(-.8,0)} \circ - []!{+(-.8,0)} \cdots ([]!{+(0,-.3)} {\underbrace{\quad \qquad \qquad}_{t\text{-vertices}}}) - []!{+(-.8,0)} \circ}$, respectively. Also, the subgraphs $R(m;t)$ denotes ${}^tL(m;t)$. 
\end{itemize}
%%%%%%%%%%%%%%%%%
Noting that $S_{\kc}$ contains exactly one singular point $p_0$, which is $\bk$-rational, by Theorem \ref{main}(2)(ii) and Lemma \ref{lem(2-2)}, the list is divided into three case about the singularity of $p_0$ (see \S\S \ref{6-1}-- \S\S \ref{6-3}, for detail): 
%%%%%%%%%%%%%%%%%%%%%%%%%%%%%%%%%%%%%%%%%%%%%%%%%%%%%%%%
\subsection{Case (I)}\label{9-1} 
%%%%%%%%%%%%%%%%%
In this case, there are $10$ cases (1)--(10): 

%%%%%%%%%%%%%%%%%
\noindent
(1) 
$\xygraph{
{^{U(t)}} - []!{+(0,-.7)} \bullet ([]!{+(.5,.3)} {\cdots},- []!{+(.5,-.5)} \circ ([]!{+(0,-.3)} {^{-(t+1)n+1}}) (- []!{+(.5,.5)} \bullet - []!{+(0,.7)} {^{U(t)}} []!{+(-.5,0)} ([]!{+(0,.4)} {\overbrace{\qquad \qquad}^{n\text{-times}}}), []!{+(1.3,0)} {(n \ge 3)}))}$
\quad
(2) 
$\xygraph{
{^{U(t)}} - []!{+(0,-.7)} \bullet ([]!{+(.5,.3)} {\cdots},- []!{+(.5,-.5)} \circ ([]!{+(0,-.3)} {^{-(t+1)n+1}}) (- []!{+(.5,.5)} \bullet - []!{+(0,.7)} {^{U(t)}} []!{+(-.5,0)} ([]!{+(0,.4)} {\overbrace{\qquad \qquad}^{n\text{-times}}}), -[]!{+(1,0)} \circ ([]!{+(0,-.3)} {^{-m}}) []!{+(.8,0)} {(n \ge 3)}))}$
\quad
(3) 
$\xygraph{
{^{U(t)}} - []!{+(0,-.7)} \bullet ([]!{+(.5,.3)} {\cdots},- []!{+(.5,-.5)} \circ ([]!{+(0,-.3)} {^{-(t+1)n+1}}) (- []!{+(.5,.5)} \bullet - []!{+(0,.7)} {^{U(t)}} []!{+(-.5,0)} ([]!{+(0,.4)} {\overbrace{\qquad \qquad}^{n\text{-times}}}), ( -[]!{+(-1,0)} A, -[]!{+(1,0)} {A^{\ast}} -[]!{+(1,0)} \circ ([]!{+(0,-.3)} {^{-m}}) []!{+(.8,0)} {(n \ge 3)})))}$

\noindent
(4) 
$\xygraph{
{^{U(t)}} - []!{+(0,-.7)} \bullet ([]!{+(.5,.3)} {\cdots},- []!{+(.5,-.5)} \circ ([]!{+(0,-.3)} {^{-(t+1)n'}}) (- []!{+(.5,.5)} \bullet - []!{+(0,.7)} {^{U(t)}} []!{+(-.5,0)} ([]!{+(0,.4)} {\overbrace{\qquad \qquad}^{n'\text{-times}}}), - []!{+(1.75,0)} \circ ([]!{+(0,-.3)} {^{-(t+1)n'}}) (- []!{+(-.5,.5)} \bullet ([]!{+(.5,.3)} {\cdots} ,- []!{+(0,.7)} {^{U(t)}} []!{+(.5,0)} ([]!{+(0,.4)} {\overbrace{\qquad \qquad}^{n'\text{-times}}})),- []!{+(.5,.5)} \bullet - []!{+(0,.7)} {^{U(t)}}))}$
\qquad
(5) 
$\xygraph{
{^{U(t)}} - []!{+(0,-.7)} \bullet ([]!{+(.5,.3)} {\cdots},- []!{+(.5,-.5)} \circ ([]!{+(0,-.3)} {^{-(t+1)n'}}) (- []!{+(.5,.5)} \bullet - []!{+(0,.7)} {^{U(t)}} []!{+(-.5,0)} ([]!{+(0,.4)} {\overbrace{\qquad \qquad}^{n'\text{-times}}}), (-[]!{+(-1.2,0)} {{}^t({A}^*)},
-[]!{+(1,0)} {{}^tA} -[r] {A}- []!{+(1,0)} \circ ([]!{+(0,-.3)} {^{-(t+1)n'}}) (-[]!{+(1.2,0)} {A^*},
(- []!{+(-.5,.5)} \bullet ([]!{+(.5,.3)} {\cdots} ,- []!{+(0,.7)} {^{U(t)}} []!{+(.5,0)} ([]!{+(0,.4)} {\overbrace{\qquad \qquad}^{n'\text{-times}}})),- []!{+(.5,.5)} \bullet - []!{+(0,.7)} {^{U(t)}}))))}$

\noindent
(6) 
$\xygraph{
{^{U(t)}} - []!{+(0,-.7)} \bullet ([]!{+(.5,.3)} {\cdots},- []!{+(.5,-.5)} \circ ([]!{+(0,-.3)} {^{-(t+1)n'-1}}) (- []!{+(.5,.5)} \bullet - []!{+(0,.7)} {^{U(t)}} []!{+(-.5,0)} ([]!{+(0,.4)} {\overbrace{\qquad \qquad}^{n'\text{-times}}}), (-[]!{+(-1.25,0)} {^{L (2;t')}},
-[]!{+(1.25,0)}\circ ([]!{+(0,-.3)} {^{-2t'-3}}) -[]!{+(1.25,0)} \circ ([]!{+(0,-.3)} {^{-(t+1)n'-1}}) (-[]!{+(1.25,0)} {^{R (2;t')}},
(- []!{+(-.5,.5)} \bullet ([]!{+(.5,.3)} {\cdots} ,- []!{+(0,.7)} {^{U(t)}} []!{+(.5,0)} ([]!{+(0,.4)} {\overbrace{\qquad \qquad}^{n'\text{-times}}})),- []!{+(.5,.5)} \bullet - []!{+(0,.7)} {^{U(t)}}))))}$

\noindent
(7) 
$\xygraph{
{^{U(t)}} - []!{+(0,-.7)} \bullet ([]!{+(.5,.3)} {\cdots},- []!{+(.5,-.5)} \circ ([]!{+(0,-.3)} {^{-(t+1)n'-1}}) (- []!{+(.5,.5)} \bullet - []!{+(0,.7)} {^{U(t)}} []!{+(-.5,0)} ([]!{+(0,.4)} {\overbrace{\qquad \qquad}^{n'\text{-times}}}), (-[]!{+(-1.25,0)} {{}^t(\underline{A^*})} -[]!{+(-1.25,0)} {^{L (m_A;t')}},
-[]!{+(1,0)} {{}^tA} -[]!{+(.8,0)}\circ ([]!{+(0,-.3)} {^{-2t'-3}}) -[]!{+(.8,0)} {A} -[]!{+(1,0)} \circ ([]!{+(0,-.3)} {^{-(t+1)n'-1}}) (-[]!{+(1.25,0)} {\underline{A^*}}-[]!{+(1.25,0)} {^{R (m_A;t')}},
(- []!{+(-.5,.5)} \bullet ([]!{+(.5,.3)} {\cdots} ,- []!{+(0,.7
)} {^{U(t)}} []!{+(.5,0)} ([]!{+(0,.4)} {\overbrace{\qquad \qquad}^{n'\text{-times}}})),- []!{+(.5,.5)} \bullet - []!{+(0,.7)} {^{U(t)}}))))}$

\noindent
(8) $\xygraph{
{{}^t({A}^*)} -[r] \bullet -[]!{+(.8,0)} {{}^tA} -[]!{+(.8,0)} {A} -[]!{+(.8,0)} \bullet -[r]{{A}^*}}$
\quad
(9) $\xygraph{
{^{L (2;t)}} -[]!{+(1,0)} \bullet -[]!{+(.8,0)} \circ ([]!{+(0,-.3)} {^{-2t-3}}) -[]!{+(.8,0)} \bullet -[]!{+(1,0)} {^{R (2;t)}}}$

\noindent
(10) $\xygraph{
{^{L (m_{A};t)}} -[]!{+(1.25,0)} {{}^t(\underline{{A}^*})} -[]!{+(1,0)} \bullet -[]!{+(.8,0)} {{}^tA} -[]!{+(.8,0)} \circ ([]!{+(0,-.3)} {^{-2t-3}}) -[]!{+(.8,0)} {A} -[]!{+(.8,0)} \bullet -[]!{+(.8,0)}{\underline{{A}^*}} -[]!{+(1.25,0)} {^{R (m_{A};t)}}}$
%%%%%%%%%%%%%%%%%%%%%%%%%%%%%%%%%%%%%%%%%%%%%%%%%%%%%%%%
\subsection{Case (II)}\label{9-2}
%%%%%%%%%%%%%%%%%
In this case, there are $17$ cases (11)--(27): 

%%%%%%%%%%%%%%%%%
\noindent
(11) 
$\xygraph{
\bullet - []!{+(.8,-.4)} \circ ([]!{+(-.1,-.3)} {^{-3}}) ( - []!{+(-.8,-.4)} \bullet ,
- []!{+(.8,0)} \circ (- []!{+(0,.8)} \circ ,- []!{+(.8,0)} \circ - []!{+(.8,0)} \circ))}$
\ 
(12)
$\xygraph{
\bullet - []!{+(.8,-.4)} \circ ([]!{+(-.1,-.3)} {^{-3}}) ( - []!{+(-.8,-.4)} \bullet ,
- []!{+(.8,0)} \circ (- []!{+(0,.8)} \circ ,- []!{+(.8,0)} \circ - []!{+(.8,0)} \circ ([]!{+(-.1,-.3)} {^{-3}})))}$
\ 
(13) 
$\xygraph{
\bullet - []!{+(.8,-.4)} \circ ([]!{+(-.1,-.3)} {^{-4}}) ((- []!{+(-.8,0)} \bullet ,- []!{+(-.8,-.4)} \bullet ),
- []!{+(.8,0)} \circ (- []!{+(0,.8)} \circ ,- []!{+(.8,0)} \circ - []!{+(.8,0)} \circ))}$

\noindent
(14)
$\xygraph{
\bullet - []!{+(.8,-.4)} \circ ([]!{+(-.1,-.3)} {^{-5}}) (((- []!{+(-.8,.15)} \bullet,- []!{+(-.8,-.15)} \bullet),- []!{+(-.8,-.4)} \bullet ),
- []!{+(.8,0)} \circ (- []!{+(0,.8)} \circ ,- []!{+(.8,0)} \circ - []!{+(.8,0)} \circ))}$
\qquad
(15) 
$\xygraph{
\circ - []!{+(.8,0)} \bullet - []!{+(.8,-.4)} \circ ([]!{+(-.1,-.3)} {^{-5}}) ( - []!{+(-.8,-.4)} \bullet - []!{+(-.8,0)} \circ ,
- []!{+(.8,0)} \circ (- []!{+(0,.8)} \circ ,- []!{+(.8,0)} \circ - []!{+(.8,0)} \circ))}$

\noindent
(16)
$\xygraph{
\bullet - []!{+(.8,-.4)} \circ ([]!{+(-.1,-.3)} {^{-3}}) ( - []!{+(-.8,-.4)} \bullet ,
- []!{+(.8,0)} \circ - []!{+(.8,0)} \circ (- []!{+(0,.8)} \circ ,- []!{+(.8,0)} \circ - []!{+(.8,0)} \circ))}$
\ 
(17) 
$\xygraph{
\bullet - []!{+(.8,-.4)} \circ ([]!{+(-.1,-.3)} {^{-3}}) ((- []!{+(-.8,0)} \circ ,- []!{+(-.8,-.4)} \bullet ),
- []!{+(.8,0)} \circ ([]!{+(-.1,-.3)} {^{-3}}) (- []!{+(0,.8)} \circ ,- []!{+(.8,0)} \circ - []!{+(.8,0)} \circ))}$
\ 
(18) 
$\xygraph{
\bullet - []!{+(.8,0)} \circ - []!{+(.8,-.4)} \circ (- []!{+(.8,0)} \circ ([]!{+(0,-.3)} {^{-m}}), - []!{+(-.8,-.4)} \circ - []!{+(-.8,0)} \bullet )}$

\noindent
(19) 
$\xygraph{
{^{L (3;t)}} - []!{+(1,0)} \bullet - []!{+(.8,0)} \circ - []!{+(.8,-.4)} \circ ([]!{+(0,-.3)} {^{-(2t+4)}})(- []!{+(.8,0)} \circ ([]!{+(0,-.3)} {^{-m}}),
- []!{+(-.8,-.4)} \circ - []!{+(-.8,0)} \bullet - []!{+(-1,0)} {^{L (3;t)}})}$
\quad
(20) 
$\xygraph{
\bullet - []!{+(.8,-.15)} \circ ([]!{+(0,-.25)} {^{-3}}) ( - []!{+(-.8,-.15)} \bullet, - []!{+(.8,-.3)} \circ (- []!{+(.8,0)} \circ ,
(- []!{+(-.8,-.3)} \circ ([]!{+(0,-.25)} {^{-3}}) (- []!{+(-.8,.15)} \bullet ,- []!{+(-.8,-.15)} \bullet ))))}$
\quad
(21) 
$\xygraph{
\circ - []!{+(.8,0)} \bullet - []!{+(.8,0)} \circ ([]!{+(-.1,-.3)} {^{-3}}) - []!{+(.8,-.4)} \circ (- []!{+(.8,0)} \circ,
- []!{+(-.8,-.4)} \circ ([]!{+(-.1,-.3)} {^{-3}}) - []!{+(-.8,0)} \bullet - []!{+(-.8,0)} \circ )}$

\noindent
(22)\!
$\xygraph{
{^{L (3;t)}} - []!{+(1,0)} \circ - []!{+(.8,0)} \bullet - []!{+(.8,0)} \circ ([]!{+(-.1,-.3)} {^{-3}}) - []!{+(.8,-.4)} \circ ([]!{+(0,-.3)} {^{-(2t+4)}})(- []!{+(.8,0)} \circ,
- []!{+(-.8,-.4)} \circ ([]!{+(-.1,-.3)} {^{-3}}) - []!{+(-.8,0)} \bullet - []!{+(-.8,0)} \circ - []!{+(-1,0)} {^{L (3;t)}})}$
\ 
(23)\!
$\xygraph{
\circ - []!{+(.8,-.15)} \circ  ( - []!{+(-.8,-.15)} \bullet, - []!{+(.8,-.3)} \circ ([]!{+(0,-.3)} {^{-4}}) (- []!{+(.8,0)} \circ ,
(- []!{+(-.8,-.3)} \circ (- []!{+(-.8,.15)} \circ ,- []!{+(-.8,-.15)} \bullet ))))}$
\ 
(24)\!
$\xygraph{
\bullet - []!{+(.8,0)} \circ - []!{+(.8,0)} \circ - []!{+(.8,-.4)} \circ (- []!{+(.8,0)} \circ,
- []!{+(-.8,-.4)} \circ - []!{+(-.8,0)} \circ - []!{+(-.8,0)} \bullet )}$

\noindent
(25) 
$\xygraph{
{^{L (4;t)}} - []!{+(1,0)} \bullet - []!{+(.8,0)} \circ - []!{+(.8,0)} \circ - []!{+(.8,-.4)} \circ ([]!{+(0,-.4)} {^{-(2t+4)}})(- []!{+(.8,0)} \circ,
- []!{+(-.8,-.4)} \circ - []!{+(-.8,0)} \circ - []!{+(-.8,0)} \bullet - []!{+(-1,0)} {^{L (4;t)}})}$
\quad
(26) 
$\xygraph{
\bullet - []!{+(.8,0)} \circ - []!{+(.8,-.6)} \circ (- []!{+(-.8,0)} \circ - []!{+(-.8,0)} \bullet ,- []!{+(-.8,-.6)} \circ - []!{+(-.8,0)} \bullet )}$
\quad
(27) 
$\xygraph{
{^{L (3;t)}} - []!{+(1,0)} \bullet - []!{+(.8,0)} \circ - []!{+(.8,-.6)} \circ ([]!{+(.2,-.3)} {^{-(3t+5)}})
(- []!{+(-.8,0)} \circ - []!{+(-.8,0)} \bullet - []!{+(-1,0)} {^{L (3;t)}},- []!{+(-.8,-.6)} \circ - []!{+(-.8,0)} \bullet - []!{+(-1,0)} {^{L (3;t)}})}$
%%%%%%%%%%%%%%%%%%%%%%%%%%%%%%%%%%%%%%%%%%%%%%%%%%%%%%%%
\subsection{Case (III)}\label{9-3}
%%%%%%%%%%%%%%%%%
In this case, there are $25$ cases (28)--(52): 

%%%%%%%%%%%%%%%%%
\noindent
(28) 
$\xygraph{
\bullet - []!{+(.8,0)} \circ - []!{+(.8,-.6)} \circ ([]!{+(0,-.3)} {^{-3}}) (- []!{+(.8,0)} \circ ,
(- []!{+(-.8,0)} \circ - []!{+(-.8,0)} \bullet ,- []!{+(-.8,-.6)} \circ - []!{+(-.8,0)} \bullet ))}$
\quad
(29) 
$\xygraph{
{^{L (3;t)}} - []!{+(1,0)} \bullet - []!{+(.8,0)} \circ - []!{+(.8,-.6)} \circ ([]!{+(.2,-.3)} {^{-(3t+6)}}) (- []!{+(.8,0)} \circ ,
(- []!{+(-.8,0)} \circ - []!{+(-.8,0)} \bullet - []!{+(-1,0)} {^{L (3;t)}},- []!{+(-.8,-.6)} \circ - []!{+(-.8,0)} \bullet - []!{+(-1,0)} {^{L (3;t)}}))}$
\quad
(30) 
$\xygraph{
\bullet - []!{+(.8,0)} \circ - []!{+(.8,-.4)} \circ ([]!{+(0,-.3)} {^{-3}})
(- []!{+(-.8,-.4)} \circ - []!{+(-.8,0)} \bullet , (- []!{+(.8,-.4)} \circ - []!{+(.8,0)} \bullet,- []!{+(.8,.4)} \circ - []!{+(.8,0)} \bullet))}$

\noindent
(31) 
$\xygraph{
{^{L (3;t)}} - []!{+(1,0)} \bullet - []!{+(.8,0)} \circ - []!{+(.8,-.4)} \circ ([]!{+(0,-.4)} {^{-(4t+7)}}) 
(- []!{+(-.8,-.4)} \circ - []!{+(-.8,0)} \bullet - []!{+(-1,0)} {^{L (3;t)}},
(- []!{+(.8,-.4)} \circ - []!{+(.8,0)} \bullet - []!{+(1,0)} {^{R (3;t)}},(- []!{+(.8,.4)} \circ - []!{+(.8,0)} \bullet - []!{+(1,0)} {^{R (3;t)}}))}$

\noindent
(32) 
$\xygraph{
{^{L (3;t)}} - []!{+(1,0)} \bullet - []!{+(.8,0)} \circ - []!{+(.8,-.4)} \circ ([]!{+(0,-.4)} {^{-(t+3)}}) 
(- []!{+(-.8,-.4)} \circ ,- []!{+(.8,0)} \circ - []!{+(.8,0)} \circ - []!{+(.8,0)} \circ ([]!{+(0,-.4)} {^{-(t+3)}})
(- []!{+(.8,-.4)} \circ,(- []!{+(.8,.4)} \circ - []!{+(.8,0)} \bullet - []!{+(1,0)} {^{R (3;t)}}))}$

\noindent
(33) 
$\xygraph{
\bullet - []!{+(.8,0)} \circ - []!{+(.8,-.4)} \circ 
(- []!{+(-.8,-.4)} \circ ,- []!{+(.8,0)} \circ ([]!{+(0,-.3)} {^{-3}}) - []!{+(.8,0)} \circ 
(- []!{+(.8,-.4)} \circ,(- []!{+(.8,.4)} \circ - []!{+(.8,0)} \bullet))}$

\noindent
(34)
$\xygraph{
{^{L (3;t)}} - []!{+(1,0)} \bullet - []!{+(.8,0)} \circ - []!{+(.8,-.4)} \circ ([]!{+(0,-.3)} {^{-(t+3)}}) 
(- []!{+(-.8,-.4)} \circ ,- []!{+(.8,0)} \circ ([]!{+(0,-.3)} {^{-3}}) - []!{+(.8,0)} \circ ([]!{+(0,-.3)} {^{-(t+3)}})
(- []!{+(.8,-.4)} \circ,(- []!{+(.8,.4)} \circ - []!{+(.8,0)} \bullet - []!{+(1,0)} {^{R (3;t)}}))}$

\noindent
(35) 
$\xygraph{
{^{L (3;t)}} - []!{+(1,0)} \bullet - []!{+(.8,0)} \circ - []!{+(.8,-.4)} \circ ([]!{+(0,-.4)} {^{-(2t+4)}}) 
(- []!{+(-.8,-.4)} \circ - []!{+(-.8,0)} \bullet - []!{+(-1,0)} {^{L (3;t)}},- []!{+(1.2,0)} \circ ([]!{+(0,-.4)} {^{-(2t+4)}})
(- []!{+(.8,-.4)} \circ - []!{+(.8,0)} \bullet - []!{+(1,0)} {^{R (3;t)}},(- []!{+(.8,.4)} \circ - []!{+(.8,0)} \bullet - []!{+(1,0)} {^{R (3;t)}}))}$
\qquad
(36) 
$\xygraph{
\bullet - []!{+(.8,-.4)} \circ ([]!{+(-.1,-.3)} {^{-6}}) ((((- []!{+(-.8,0)} \bullet,- []!{+(-.8,.2)} \bullet,- []!{+(-.8,-.2)} \bullet),- []!{+(-.8,-.4)} \bullet )),
- []!{+(.8,0)} \circ (- []!{+(0,.8)} \circ ,- []!{+(.8,0)} \circ - []!{+(.8,0)} \circ))}$

\noindent
(37)
$\xygraph{
\bullet - []!{+(.8,-.15)} \circ  ([]!{+(-.1,-.25)} {^{-4}}) ((- []!{+(-.8,0)} \bullet, - []!{+(-.8,-.15)} \bullet), - []!{+(.8,-.25)} \circ (- []!{+(.8,0)} \circ,
(- []!{+(-.8,-.25)} \circ  ([]!{+(-.1,-.25)} {^{-4}}) (- []!{+(-.8,0)} \bullet,(- []!{+(-.8,.15)} \bullet ,- []!{+(-.8,-.15)} \bullet )))))}$
\qquad
(38)
$\xygraph{
\circ - []!{+(.8,0)} \circ - []!{+(.8,0)} \bullet - []!{+(.8,0)} \circ ([]!{+(-.1,-.3)} {^{-4}}) - []!{+(.8,-.4)} \circ (- []!{+(.8,0)} \circ ,
- []!{+(-.8,-.4)} \circ ([]!{+(-.1,-.3)} {^{-4}}) - []!{+(-.8,0)} \bullet - []!{+(-.8,0)} \circ - []!{+(-.8,0)} \circ)}$

\noindent
(39)
$\xygraph{
{^{L (3;t)}} - []!{+(1,0)} \circ - []!{+(.8,0)} \circ - []!{+(.8,0)} \bullet - []!{+(.8,0)} \circ ([]!{+(-.1,-.3)} {^{-4}}) - []!{+(.8,-.4)} \circ ([]!{+(0,-.3)} {^{-(2t+4)}})(- []!{+(.8,0)} \circ ,
- []!{+(-.8,-.4)} \circ ([]!{+(-.1,-.3)} {^{-4}}) - []!{+(-.8,0)} \bullet - []!{+(-.8,0)} \circ - []!{+(-.8,0)} \circ - []!{+(-1,0)} {^{L (3;t)}})}$
\qquad
(40) 
$\xygraph{
\circ - []!{+(.8,-.15)} \circ  ( - []!{+(-.8,-.15)} \bullet, - []!{+(.8,-.3)} \circ ([]!{+(0,-.3)} {^{-4}}) (- []!{+(.8,0)} \circ ([]!{+(0,-.3)} {^{-3}}),
(- []!{+(-.8,-.3)} \circ (- []!{+(-.8,.15)} \circ ,- []!{+(-.8,-.15)} \bullet ))))}$

\noindent
(41) 
$\xygraph{
\bullet - []!{+(.8,0)} \circ - []!{+(.8,0)} \circ - []!{+(.8,-.4)} \circ (- []!{+(.8,0)} \circ ([]!{+(-.1,-.3)} {^{-3}}),
- []!{+(-.8,-.4)} \circ - []!{+(-.8,0)} \circ - []!{+(-.8,0)} \bullet)}$
\qquad
(42) 
$\xygraph{
{^{L (4;t)}} - []!{+(1,0)} \bullet - []!{+(.8,0)} \circ - []!{+(.8,0)} \circ - []!{+(.8,-.4)} \circ ([]!{+(0,-.4)} {^{-(2t+4)}})(- []!{+(.8,0)} \circ ([]!{+(-.1,-.3)} {^{-3}}),
- []!{+(-.8,-.4)} \circ - []!{+(-.8,0)} \circ - []!{+(-.8,0)} \bullet - []!{+(-1,0)} {^{L (4;t)}})}$

\noindent
(43) 
$\xygraph{
\circ - []!{+(.8,0)} \circ - []!{+(.8,-.15)} \circ  ( - []!{+(-.8,-.15)} \bullet, - []!{+(.8,-.3)} \circ ([]!{+(0,-.3)} {^{-6}}) (- []!{+(.8,0)} \circ ,
(- []!{+(-.8,-.3)} \circ (- []!{+(-.8,.15)} \circ - []!{+(-.8,0)} \circ ,- []!{+(-.8,-.15)} \bullet ))))}$
\qquad
(44) 
$\xygraph{
{^{L (5;t)}} - []!{+(1,0)} \bullet - []!{+(.8,0)} \circ - []!{+(.8,0)} \circ - []!{+(.8,0)} \circ - []!{+(.8,-.4)} \circ ([]!{+(0,-.4)} {^{-(2t+4)}})(- []!{+(.8,0)} \circ,
- []!{+(-.8,-.4)} \circ - []!{+(-.8,0)} \circ - []!{+(-.8,0)} \circ - []!{+(-.8,0)} \bullet - []!{+(-1,0)} {^{L (5;t)}})}$

\noindent
(45) 
$\xygraph{
\bullet - []!{+(.8,-.15)} \circ  ([]!{+(-.1,-.25)} {^{-3}}) ( - []!{+(-.8,-.15)} \bullet, - []!{+(.8,-.3)} \circ (- []!{+(.8,0)} \circ  ([]!{+(0,-.3)} {^{-3}}),
(- []!{+(-.8,-.3)} \circ  ([]!{+(-.1,-.25)} {^{-3}}) (- []!{+(-.8,.15)} \bullet ,- []!{+(-.8,-.15)} \bullet ))))}$
\quad
(46) 
$\xygraph{
\circ - []!{+(.8,0)} \bullet - []!{+(.8,0)} \circ ([]!{+(-.1,-.3)} {^{-3}}) - []!{+(.8,-.4)} \circ (- []!{+(.8,0)} \circ ([]!{+(-.1,-.3)} {^{-3}}),
- []!{+(-.8,-.4)} \circ ([]!{+(-.1,-.3)} {^{-3}}) - []!{+(-.8,0)} \bullet - []!{+(-.8,0)} \circ)}$

\noindent
(47) 
$\xygraph{
{^{L (3;t)}} - []!{+(1,0)} \circ - []!{+(.8,0)} \bullet - []!{+(.8,0)} \circ ([]!{+(-.1,-.3)} {^{-3}}) - []!{+(.8,-.4)} \circ ([]!{+(0,-.3)} {^{-(2t+4)}})(- []!{+(.8,0)} \circ ([]!{+(0,-.3)} {^{-3}}),
- []!{+(-.8,-.4)} \circ ([]!{+(-.1,-.3)} {^{-3}}) - []!{+(-.8,0)} \bullet - []!{+(-.8,0)} \circ - []!{+(-1,0)} {^{L (3;t)}})}$
\quad
(48) 
$\xygraph{
\bullet - []!{+(.8,-.15)} \circ ([]!{+(-.1,-.3)} {^{-3}}) ( - []!{+(-.8,-.15)} \bullet, - []!{+(.8,-.6)} \circ
(- []!{+(-.8,0)} \circ ([]!{+(-.1,-.3)} {^{-3}}) (- []!{+(-.8,.15)} \bullet,- []!{+(-.8,-.15)} \bullet ),- []!{+(-.8,-.6)} \circ ([]!{+(-.1,-.3)} {^{-3}}) (- []!{+(-.8,.15)} \bullet,- []!{+(-.8,-.15)} \bullet )))}$
\quad
(49) 
$\xygraph{
\circ - []!{+(.8,0)} \bullet - []!{+(.8,0)} \circ ([]!{+(-.1,-.3)} {^{-3}}) - []!{+(.8,-.6)} \circ
(- []!{+(-.8,0)} \circ ([]!{+(-.1,-.3)} {^{-3}}) - []!{+(-.8,0)} \bullet - []!{+(-.8,0)} \circ,- []!{+(-.8,-.6)} \circ ([]!{+(-.1,-.3)} {^{-3}}) - []!{+(-.8,0)} \bullet - []!{+(-.8,0)} \circ)}$

\noindent
(50)
$\xygraph{
{^{L (3;t)}} - []!{+(1,0)} \circ - []!{+(.8,0)} \bullet - []!{+(.8,0)} \circ ([]!{+(-.1,-.3)} {^{-3}}) - []!{+(.8,-.6)} \circ ([]!{+(.2,-.3)} {^{-(3t+5)}})
(- []!{+(-.8,0)} \circ ([]!{+(-.1,-.3)} {^{-3}}) - []!{+(-.8,0)} \bullet - []!{+(-.8,0)} \circ - []!{+(-1,0)} {^{L (3;t)}},- []!{+(-.8,-.6)} \circ ([]!{+(-.1,-.3)} {^{-3}}) - []!{+(-.8,0)} \bullet - []!{+(-.8,0)} \circ - []!{+(-1,0)} {^{L (3;t)}})}$
\ 
(51)
$\xygraph{
\circ - []!{+(.8,-.15)} \circ ( - []!{+(-.8,-.15)} \bullet, - []!{+(.8,-.6)} \circ ([]!{+(-.1,-.3)} {^{-5}})
(- []!{+(-.8,0)} \circ (- []!{+(-.8,.15)} \circ,- []!{+(-.8,-.15)} \bullet ),- []!{+(-.8,-.6)} \circ (- []!{+(-.8,.15)} \circ,- []!{+(-.8,-.15)} \bullet )))}$
\ 
(52) 
$\xygraph{
{^{L (4;t)}} - []!{+(1,0)} \bullet - []!{+(.8,0)} \circ - []!{+(.8,0)} \circ - []!{+(.8,-.6)} \circ ([]!{+(.2,-.3)} {^{-(3t+5)}})
(- []!{+(-.8,0)} \circ - []!{+(-.8,0)} \circ - []!{+(-.8,0)} \bullet - []!{+(-1,0)} {^{L (4;t)}},- []!{+(-.8,-.6)} \circ - []!{+(-.8,0)} \circ - []!{+(-.8,0)} \bullet - []!{+(-1,0)} {^{L (4;t)}})}$
%%%%%%%%%%%%%%%%%%%%%%%%%%%%%%%%%%%%%%%%%%%%%%%%%%%%%%%%%%%%%%%%%%%%%%%%%%%%%%%%%%%%%%%%%%%%%%%%%%%%%%%%%%

\end{document}